\pdfoutput=1

\documentclass[paper=a4, fontsize=11pt, twoside]{scrartcl}
\usepackage[inner=3cm, outer=2cm]{geometry}
\usepackage[T1]{fontenc}

\usepackage[english]{babel}	
\usepackage[protrusion=true,expansion=true]{microtype}	
\usepackage{amsmath,amsfonts,amsthm,amssymb,mathtools} 
\usepackage[pdftex]{graphicx}	
\usepackage{url}
\usepackage{tikz-cd}
\usepackage[nomessages]{fp}

\usepackage[style=alphabetic, maxnames=5, maxalphanames=5]{biblatex}
\renewbibmacro{in:}{}
\DeclareFieldFormat{pages}{#1}
\usepackage[absolute]{textpos}
\usepackage{sectsty}
\allsectionsfont{\centering \normalfont\scshape}

\usepackage{fancyhdr}
\usepackage[strict]{changepage}
\usepackage{pdflscape}
\usepackage{dsfont}
\usepackage{braket}
\usepackage{commath}

\usepackage{hyperref}
\usepackage{cleveref}

\newtheorem{theorem}{Theorem}
\newtheorem{proposition}[theorem]{Proposition}
\newtheorem{lemma}[theorem]{Lemma}

\newtheorem{hypothesis}[theorem]{Hypothesis}
\theoremstyle{definition}
\newtheorem{defn}[theorem]{Definition}

\newtheorem{rem}[theorem]{Remark}

\newtheorem{app}[theorem]{Application}
\newtheorem{conv}[theorem]{Convention}

\numberwithin{equation}{section}		
\numberwithin{figure}{section}		
\numberwithin{table}{section}

\newcommand{\Left}[2]{\left#1 \vphantom{\del[#2]{X^X}}\right.\!}
\newcommand{\Right}[2]{\!\left. \vphantom{\del[#2]{X^X}}\right#1}

\newcommand{\reals}{\mathbb{R}}

\newcommand{\ints}{\mathbb{Z}}
\newcommand{\nats}{\mathbb{N}}
\newcommand{\field}{\mathbb{K}}
\newcommand{\id}{\mathrm{id}}

\newcommand{\one}{\mathds{1}}

\newcommand{\bord}[1]{{\mathcal{#1}}}

\newcommand{\homsetin}[3]{\mathrm{Hom}_{#1} \left( #2 ,\, #3 \right)}

\newcommand{\vect}{\mathrm{vect}}

\newcommand{\bordcat}{\mathrm{Cob}_3^\ccat}

\newcommand{\bordcatskel}{\mathrm{Cob}_{3,\mathrm{sk}}^\ccat}

\newcommand{\ccat}{\mathfrak{C}}
\newcommand{\pdim}[1]{\mathrm{d}_{#1}}

\newcommand{\pic}[2]{\FPeval{\result}{round(0.7*#1,5)}\raisebox{-0.5\height}{\includegraphics[height=\result cm]{#2}}}
\newcommand{\pica}[2]{\raisebox{-0.5\height}{\includegraphics[scale=#1]{#2}}}

\renewcommand{\Left}[2]{\left#1 \vphantom{\del[#2]{X^X}}\right.\!}
\renewcommand{\Right}[2]{\!\left. \vphantom{\del[#2]{X^X}}\right#1}

\begin{document}
\begin{textblock}{10}(8,1)
\begin{flushright}
\noindent\small[ZMP-HH/21-5]\\
Hamburger Beitr\"age zur Mathematik Nr. 893\\
April 2021
\end{flushright}
\end{textblock}
\title{
		\large Frobenius-Schur Indicators and the Mapping Class Group of the Torus
}
\author{
		\normalfont 								\normalsize
        Julian Farnsteiner and Christoph Schweigert\\		\normalsize
        \small{\textit{Fachbereich Mathematik}}\\[-5pt]
        \small{\textit{Universit\"at Hamburg}}\\[-5pt]
        \small{\textit{Bereich Algebra und Zahlentheorie}}\\[-5pt]
        \small{\textit{Bundesstra\ss{}e 55}}\\[-5pt]
        \small{\textit{D – 20146 Hamburg}}
}
\date{}

\maketitle

\begin{abstract}
The Turaev-Viro state sum invariant can be extended to 3-manifolds with
free boundaries. We use this fact to describe generalized Frobenius-Schur indicators
as Turaev-Viro invariants of solid tori. This provides
a geometric understanding of the $\mathrm{SL}(2,\ints)$-equivariance of these indicators. 
\end{abstract}

\tableofcontents

\section{Introduction}\label{sec_intro}
Let $\field$ be an algebraically closed field.
The Frobenius-Schur indicator assigns a scalar $\nu(V) \in \field$ to a finite-dimensional $\field$-linear representation $V$ of a finite group $G$.
The concept of a Frobenius-Schur indicator has found its natural conceptual home in the context of spherical fusion categories \cite{ksz_higher_fs,sz-congruencesubgroup} \cite{ns-basic,ns-general}. In full generality, given a fixed spherical fusion category $\ccat$ over $\field$, the Frobenius-Schur indicator $\nu_{n,r}^X(V) \in \field$ is a scalar assigned to an object $V \in \ccat$, with an additional argument in the Drinfeld center $X \in Z(\ccat)$, and $(n,r)\in \ints \times \ints$.
Specializing to the case $\ccat = G$-rep, $X = \one$, $n=2$, $r=1$, one obtains the classical indicator $\nu(V) = \nu_{2,1}^\one(V)$.
We recall the definition in \Cref{sec_cats} (\Cref{defn_fs_general}).

In the upper argument $X$, the Frobenius-Schur indicator is additive with respect to direct sums. 
This allows us to extend the range of this argument from objects $X \in Z(\ccat)$ in the Drinfeld center to elements of the Grothendieck algebra of the Drinfeld center, $K_0(Z(\ccat))\otimes_\ints \field$. 
The Drinfeld center $Z(\ccat)$ is a modular fusion category, hence the $\field$-vector space $K_0(Z(\ccat))\otimes_\ints \field$ comes with the structure of a left action of the modular group $\mathrm{SL}(2,\ints)$, see \cite[Sec. 1.4]{ns-general}.
The group $\mathrm{SL}(2,\ints)$ also acts on the lower argument $(n,r)$ of the indicator via right matrix multiplication.

The Frobenius-Schur indicator turns out to satisfy a rather surprising equivariance property with respect to these actions, as first shown in \cite{sz-congruencesubgroup}: 
For $\mathfrak{g} \in \mathrm{SL}(2,\ints)$, we have
\begin{equation}\label{eq_intro_equiv}
    \nu_{(n,r)}^{\mathfrak{g}[X]} (V) = \nu_{(n,r)\Tilde{\mathfrak{g}}}^{[X]} (V),
\end{equation}
where $\Tilde{\mathfrak{g}} = \mathrm{diag}(1,-1)\;\mathfrak{g}\;\mathrm{diag}(1,-1)$.

In this general form, Frobenius-Schur indicators have found several applications.
\begin{app}
Kashina, Sommerh\"auser and Zhu extensively study higher Frobenius-Schur indicators in representation categories of semisimple Hopf algebras in \cite{ksz_higher_fs}. (Note that the representation category of a semisimple Hopf algebra has a canonical pivotal structure.) In their work, three
distinct formulae for the indicators are presented. Their results include a generalization of
the classical theorem on Frobenius-Schur indicators \cite[Thm. 2.4]{ksz_higher_fs}, as well as relations to other invariants, among them the following (previously conjectured) result: 
If $p\in\field$ is a prime in an algebraically closed field of characteristic zero, and $p$ divides $\dim H$ (where $H$ is a semisimple finite-dimensional Hopf algebra over $\field$), then $p$ also divides the so-called exponent $\exp H$ of $H$ \cite[Thm. 3.4]{ksz_higher_fs}. 
\end{app}

\begin{app}
In \cite{ns-quasihopf}, Ng and Schauenburg apply the definition of higher Frobenius-Schur indicators for $X=\one$ to the case of semisimple quasi-Hopf algebras.

They present four quasi-Hopf algebras whose representation categories share the same number of (isomorphism classes of) simple objects and the same fusion rules, and use the indicators to show that nevertheless, these categories are inequivalent as monoidal categories.
Before their work, it had been known that precisely four inequivalent fusion categories of this type exist, but an explicit construction in terms of representation categories had not been presented.

Different quasi-Hopf algebras with tensor-equivalent representation categories share the same (higher) Frobenius-Schur indicators.
Ng and Schauenburg establish inequivalence by finding a family of representations -- one for each of the four quasi-Hopf algebras -- that are necessarily identified by any tensor equivalence.
They then explicitly compute some indicators and show that they do not agree \cite[Thm. 6.1]{ns-quasihopf}.
To this end, the second (classical) Frobenius-Schur indicator is not sufficient and higher indicators are needed.
\end{app}

\begin{app}
Finally, we remark that the full generality of \Cref{defn_fs_general} has been used to prove the congruence subgroup conjecture, in \cite[Thms. 9.3 and 9.4]{sz-congruencesubgroup} for semisimple Hopf algebras, and in \cite[Thms. 6.7 and 6.8]{ns-general} in the context of spherical fusion categories, along with some conjectures about rational conformal field theory \cite[Sec. 8]{ns-general}.
In this application, it is important to consider for $X \in Z(\ccat)$ also objects different from the monoidal unit.
\end{app}

While the equivariance \eqref{eq_intro_equiv} has proven to be useful, its algebraic proofs in \cite{sz-congruencesubgroup} and \cite{ns-general} only provide limited insight into its conceptual origin.
The purpose of this note is to provide such an understanding in topological terms.

The basic idea is to obtain the Frobenius-Schur indicator from the evaluation of a topological field theory (TFT) associated to the spherical fusion category $\ccat$ on a suitable 3-manifold. 
Concretely, we build on the construction of a state-sum theory in the formulation of Turaev and Virelizier \cite{tv}. 
An important algebraic feature of the Frobenius-Schur indicator is the fact that both the category $\ccat$ and its Drinfeld center $Z(\ccat)$ appear.
The Drinfeld center naturally appears in Turaev-Viro invariants of closed oriented 3-manifolds.
To see also objects in $\ccat$ itself, we extend the notion of Turaev-Viro invariants to 3-manifolds with boundary. 
Note that in contrast to the manifolds with boundary that most commonly appear in this context (e.g. in [TV17]), our boundaries are not ``gluing'' boundaries that arise when cutting a larger 3-manifold into pieces, but ``physical'' or \emph{free} boundaries.
Consequently, a scalar (not a linear map, as one might expect) is assigned to a 3-manifold with free boundary. 
As usual for state-sum constructions, auxiliary structure is needed.
Following \cite{tv}, we use a skeleton embedded in the manifold (similar to a triangulation). 
We describe this construction in \Cref{sec_tvgen}.
In \Cref{sec_tvfs}, we show that the Frobenius-Schur indicator is obtained for a specific skeleton from a solid torus with additional structure, namely embedded links (Wilson lines) in the interior and on the boundary.
This is the content of \Cref{prop_tvfs}.
An example for such a solid torus is given in picture \eqref{eq_fs_man_ex}, the skeleton we use
in \eqref{eq_fs_man_skel}.

In this setup, the $\mathrm{SL}(2,\ints)$-equivariance can be related to geometry: $\mathrm{SL}(2,\ints)$ is the mapping class group of the two-dimensional torus. 
Using standard techniques from topological field theory, we can cut the solid torus along an embedded torus into two connected components, and glue them together again along an element of the mapping class group. 
This is illustrated in the picture \eqref{glue_tor_inside} in \Cref{sec_tvfs_equiv}.  
We show that this operation amounts to acting with the corresponding element of $\mathrm{SL}(2,\ints)$ on the upper argument of the Frobenius-Schur indicator; it also amounts to acting on the lower argument.
This fact implies the equivariance \eqref{eq_intro_equiv}.

Our construction uses explicit skeleta. 
The arguments can be extended to an alternative proof of $\mathrm{SL}(2,\ints)$-equivariance. However, such a proof requires a lot of elaboration and technical arguments, in particular about skeleta, that do not add to the conceptual understanding of the well-established equivariance property (see \Cref{rem_noproof}). 
It is therefore much more attractive to complete a proof by showing that our construction of assigning scalars to 3-manifolds with boundary, together with the usual Turaev-Viro theory, assemble into a TFT defined on a cobordism category with free boundaries.
Such a proof, unfortunately, is beyond the scope of this note and is relegated to future
work. 
From this perspective, the construction in this note may be seen as a consistency check that it makes sense to generalize the Turaev-Viro TFT to cobordisms with free boundary.

\section{Preliminaries}
\label{sec_cats}
Throughout this work, let $\ccat$ denote a spherical fusion category over an algebraically closed field $\field$ of characteristic 0.
We follow the conventions in \cite{egno}, and proceed to make some definitions more explicit and introduce some notation.

We write the data of a monoidal category as $(\ccat, \otimes, \one)$, or simply as $\ccat$, consistently suppressing associativity and unitality constraints in notation.

For a monoidal category $\ccat$, we use the following notion of duality:
A \emph{right dual} of an object $X\in \ccat$ is an object $X^\vee \in \ccat$, together with morphisms $\mathrm{ev}_X : X^\vee \otimes X \to \one$ and $\mathrm{coev}_X : \one \to X \otimes X^\vee$, satisfying zig-zag relations.
In the same situation, $X$ together with these morphisms is called a \emph{left dual} of $X^\vee$.

A basic example for a category with duals, which we recall in order to introduce a notational convention, is the category $\vect$ of finite-dimensional $\field$-vector spaces.
A right dual of a (finite-dimensional) vector space $V$ is given by the linear dual $V^*$, with the evaluation given by the standard pairing.
In terms of a basis $(\alpha_i)$ and the corresponding dual basis $(\alpha^*_i)$ the coevaluation in $V$ is defined by
\begin{equation}\label{eq_coev_sum_conv}
    \mathrm{coev}_V (1) = \sum_{i=1}^{\dim(V)}\alpha_i \otimes \alpha^*_i = \alpha_i \otimes \alpha^*_i = \alpha \otimes \alpha^*.
\end{equation}
We here introduced a form of the Einstein summation convention, dropping the sum from the notation, or even the summation index. 
This convention will in particular be applied to hom-spaces.

A pivotal structure on a rigid monoidal category $\ccat$ is a monoidal natural isomorphism $j: \id_\ccat \to (-)^{\vee\vee}$.
In a pivotal category $\ccat$, that is, a rigid monoidal category with pivotal structure, we can define left and right traces of endomorphisms $f$, $\mathrm{tr}_l (f)$ and $\mathrm{tr}_r (f)$.
These are morphisms $\one \to \one$, which, if $\ccat$ is a $\field$-linear category with simple unit $\one$, correspond under the canonical identification $\homsetin{\ccat}{\one}{\one}\cong \field$ to a scalar.
The left and right (pivotal) dimensions $\dim_l(X)$ and $\dim_r(X)$ of an object $X \in \ccat$ are given by the corresponding traces of the identity.
A pivotal structure is said to be spherical if $\mathrm{tr}_r (f) = \mathrm{tr}_l (f)$ for any endomorphism $f$ of any object $X$.
In this case, we drop the subscripts $r$ and $l$ from the expressions for traces and dimensions, and abbreviate further $\pdim{X} = \dim(X)$.

There is a coherence theorem for pivotal monoidal categories \cite[Thm. 2.2]{ns-basic}, which states that any pivotal monoidal category is equivalent (under a suitable notion of pivotal monoidal equivalence) to a strict pivotal monoidal category, in which the tensor product is strict and $j$ is the identity.
This justifies omitting the pivotal constraint $j$ from our notation when we work with these categories. 
Moreover, we do not distinguish right duals and left duals, and write in pivotal categories $X^*$ for the two-sided dual of an object $X$.\\

In a spherical fusion category $\ccat$, the dimensions of simple objects are non-zero \cite[Prop. 4.8.4]{egno}.
We will always denote by $I$ a complete set of representatives of simple objects, and assume that the representative of $[\one]$ is $\one$.
We define the dimension of the category as 
\begin{equation*}
    \dim(\ccat) := \sum_{i\in I} \pdim{i}^2.
\end{equation*}
Note that some authors \cite{balskir} define the dimension to be the square root of this quantity instead.

Moreover,  the pairing
\begin{equation}\label{eq_pairing}
    \homsetin{\ccat}{X}{Y}\otimes\homsetin{\ccat}{Y}{X} \to \one, \quad f\otimes g \mapsto \mathrm{Tr}(g\circ f)
\end{equation}
is non-degenerate (semisimplicity is essential for this), inducing an isomorphism 
\begin{equation}\label{eq_homspacedual_serre}
    \homsetin{\ccat}{X}{Y}^* \cong \homsetin{\ccat}{Y}{X}.
\end{equation}

We will make heavy use of the graphical calculus of string diagrams appropriate for spherical fusion categories, with conventions similar to those in e.g. \cite{balskir}.
Recall that the principal feature of string diagrams in \emph{spherical} fusion categories is that closed diagrams on the sphere $\mathbb{S}^2$ can be evaluated to a scalar in $\field$ (see \cite[Lem. 2.9]{tv}).

We will often consider pairs of dual bases of a hom-space. 
These will be denoted by lowercase greek letters $\alpha, \beta, \dots$.
We adopt the summation conventions from \eqref{eq_coev_sum_conv} to the graphical calculus:
Whenever we draw a string diagram in which a coupon is labeled by $\alpha$ (or $\beta, \dots$), and another coupon is labeled by $\alpha^*$, we evaluate it as follows.
Choose a basis $(\alpha_i)$ of a hom-set appropriate for the $\alpha$-labeled coupon, and denote by $(\alpha_i^*)$ the image of the dual basis under the isomorphism \eqref{eq_homspacedual_serre}.
We require that the adjacent strands of the $\alpha^*$-coupon are such that the the coupon is allowed to be labeled by the elements $\alpha_i^*$.
Then we replace the labels $\alpha$ and $\alpha^*$ by $\alpha_i$ and $\alpha_i^*$, and take the sum over $i$.
As an example, the following lemma uses this convention:

\begin{lemma}\label{prop_trace_homendos}
Let $X,Y \in \ccat$, and $F : \homsetin{\ccat}{X}{Y} \to \homsetin{\ccat}{X}{Y}$ a linear map. 
Then the following equality holds for the trace $\mathrm{Tr}$ of linear endomorphisms of vector spaces:
\begin{equation*}
    \mathrm{Tr}(F) = \pica{1}{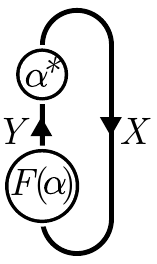}.
\end{equation*}
\end{lemma}
\begin{proof}
The string diagram depicted involves a choice of basis $(\alpha_i)$ of $\homsetin{\ccat}{X}{Y}$ (but we will see that its value in $\field$ is independent of this choice).
The identification \eqref{eq_homspacedual_serre} determines a dual basis $(\alpha_i^*)$ of $\homsetin{\ccat}{Y}{X}$.
The basis $(\alpha_i)$ also determines a matrix representing $F$, whose components we denote by $F_{ij}$.
The diagram is then evaluated as
\begin{equation*}
    \sum_i \mathrm{tr}(\alpha_i^* \circ F(\alpha_i)) = \sum_i \mathrm{tr}\del[2]{\alpha_i^* \circ \sum_j F_{ji}\alpha_j} = \sum_{ij} F_{ji} \underbrace{\mathrm{tr}(\alpha_i^* \circ \alpha_j)}_{\delta_{ij}} = \sum_i F_{ii} = \mathrm{Tr}(F).
\end{equation*}
\end{proof}

We recall some more rules for string diagrams in spherical fusion categories.
In the lemmata below, single strands labeled by capital letters $X,Y,\dots$ can be replaced by a collection of parallel strands. 
$\Gamma$-labeled coupons are always meant to be a suitable morphism.

In particular, string diagrams with coupons labeled by pairs of dual bases come with useful rules that allow us to simplify them.
The following lemma shows that we can dissolve strands. A detailed proof can be found in \cite[Lem. 1.8]{passegger}.
It is clear that finiteness of the category $\ccat$ is necessary to state the lemma, and semisimplicity is required in the proof.
\begin{lemma}[{\cite[Lem. 1.1 (2)]{balskir}}]\label{rule_dissolve_edges}
Recall that $I$ is a set of representatives of simple objects in $\ccat$.
Then
\begin{equation*}
    \sum_{i\in I} \pdim{i} \pica{1}{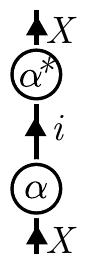} \quad = \quad  \pica{1}{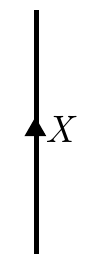}.
\end{equation*}
\end{lemma}

In other instances, disconnected subgraphs can be glued together along a pair of dual coupons:
\begin{lemma}[{\cite[Lem. 1.3]{balskir}}]\label{rule_dissolve_verts}
We have
\begin{equation*}
    \pica{1}{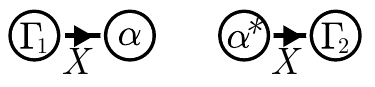} \qquad = \qquad \pica{1}{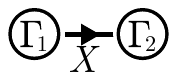}.
\end{equation*}
Note that the coupons $\Gamma_1$, $\Gamma_2$ can themselves be replaced by string diagrams, making this a statement about two disconnected closed diagrams, each having a coupon labeled by one of two bases that are dual to one another.
\end{lemma}
\begin{proof}
This is a corollary of \Cref{rule_dissolve_edges}.
By Schur's Lemma, any morphism $\one \to i$, for $i\neq \one$ simple, is zero.
Since $\ccat$ is a fusion category, $\one$ is simple, so in the sum
\begin{equation*}
    \sum_i \pdim{i} \pica{1}{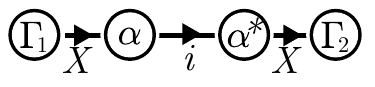} \quad = \quad  \pica{1}{chap_cats_pics/glue_graphs_after.pdf}
\end{equation*}
(applying \Cref{rule_dissolve_edges}), only the term $i = \one$ survives.
\end{proof}

\begin{lemma}\label{rule_mergebases}
Let $X,Y\in \ccat$ and let $(\varphi_{i,\alpha})_{\alpha=1\dots [i:X]}$ be, for each $i\in I$, a basis of $\homsetin{\ccat}{i}{X}$, and similarly $(\psi_{i,\alpha})_{\alpha=1\dots [i:Y]}$ be a basis of $\homsetin{\ccat}{i}{Y}$.
Denote by $\varphi_i^\alpha$, $\psi_i^\alpha$ the elements of the dual bases.
Then the maps 
\begin{equation*}
    \left( \psi_{i,\alpha} \circ \varphi_i^\beta \right)_{i\in I, \alpha=1\dots [i:X], \beta=1\dots [i:Y]}
\end{equation*}
form a basis of $\homsetin{\ccat}{X}{Y}$, and the dual basis of $\homsetin{\ccat}{Y}{X}$ is given by
\begin{equation*}
    \left( d_i\; \varphi_{i,\beta} \circ \psi_i^\alpha \right)_{i\in I, \alpha=1\dots [i:X], \beta=1\dots [i:Y]}.
\end{equation*}
This implies the following rule for string diagrams, for any morphism $\Gamma$:
\begin{equation*}
    \sum_i d_i \pica{1}{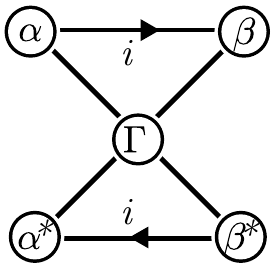} = \pica{1}{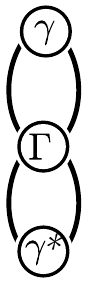}.
\end{equation*}
\end{lemma}
\begin{proof}
This follows by a direct computation using Schur's lemma.
\end{proof}

\begin{lemma}\label{rule_tracedbases}
Similarly, we have the following rule for all endomorphisms $\Gamma$ of any object $X\in \ccat$.
 \begin{equation*}
     \sum_i d_i \quad\pica{1}{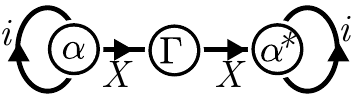} \quad=\quad \dim(\ccat) \quad \pica{1}{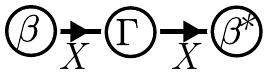}.
 \end{equation*}
\end{lemma}
\begin{proof}
By \Cref{rule_dissolve_verts}, the left-hand side is equal to
 \begin{equation*}
     \sum_i d_i \quad\pica{1}{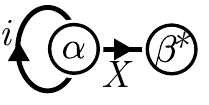}\;\pica{1}{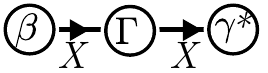}\;\pica{1}{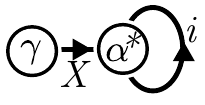}.
 \end{equation*}
We again use \Cref{rule_dissolve_verts} to eliminate the vertices labeled by $\alpha$. 
In this way, the strands labeled by $i$ form a circle, which brings an additional factor $d_i$ to the term:
  \begin{equation*}
     \sum_i d_i^2 \quad\pica{1}{letter_pics/tracedbases_c2.pdf}\;\pica{1}{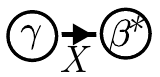}.
 \end{equation*}
After applying \Cref{rule_dissolve_verts} once more to dissolve the coupons labeled by $\beta$, we obtain the right-hand side in \Cref{rule_tracedbases}.
\end{proof}

Coupons labeled by a pair of dual bases will appear excessively in the string diagrams we consider, so we will employ the following convention:
\begin{conv}\label{conv_dual_verts_in_stringdiags}
When drawing a string diagram with some unlabeled vertices, we implicitly label each pair of dual vertices by a pair of dual bases and sum over basis elements, as in the following example, where a sum over $\alpha$ is implied:
\begin{equation*}
    \pica{1}{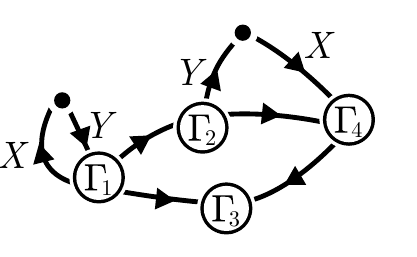} = \pica{1}{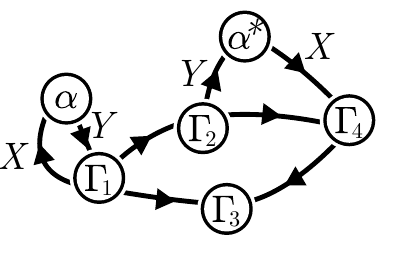}.
\end{equation*}
Of course, this only makes sense if the labeling of the strands uniquely determines which vertices are labeled by dual morphisms.
\end{conv}

The equivariant Frobenius-Schur indicators in the sense of Kashina, Sommerh\"auser and Zhu \cite{ksz_higher_fs, sz-congruencesubgroup} and Ng and Schauenburg  \cite{ns-basic,ns-exponents,ns-quasihopf,ns-general}, can be conveniently expressed using the graphical calculus in $\ccat$:

\begin{defn}[{\cite[Def. 2.1]{ns-general}}] \label{defn_fs_general}
Let $\ccat$ be a spherical fusion category, and let $X \in Z(\ccat)$ be an object in the Drinfeld center of $\ccat$. 
Then for $n \in \nats$, $r \in \{0,\dots,n\}$, the $r$-twisted $n$-th (generalized) Frobenius-Schur indicator of an object $V \in \ccat$ is given by the scalar
\begin{equation*}
    \nu_{n,r}^X(V) = \mathrm{Tr} \left(\pic{9}{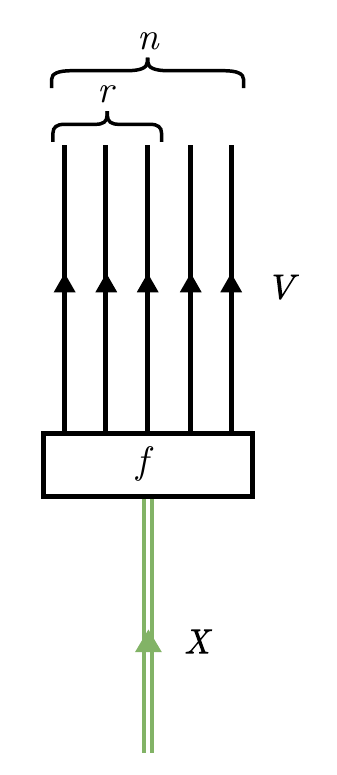} \qquad \mapsto \qquad \pic{9}{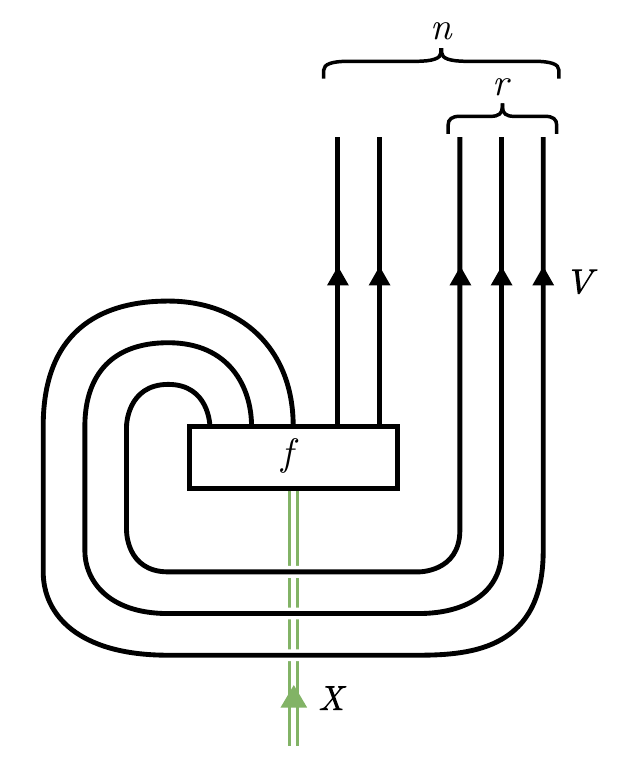} \right) \in \field,
\end{equation*}
where $f\in \homsetin{\ccat}{X}{V^{\otimes n}}$.
\end{defn}

The Frobenius-Schur indicator is thus the trace over an automorphism $E^{X,V}_{n,r} : \homsetin{\ccat}{X}{V^{\otimes n}} \to \homsetin{\ccat}{X}{V^{\otimes n}}$, where we suppressed the forgetful functor $Z(\ccat)\to \ccat$, applied to $X$, in the notation.
These pictures establish the use of double-stroke green lines for objects in the Drinfeld center. 
Crossings of strands are then to be interpreted using the half-braiding of $X\in Z(\ccat)$.
For an arbitrary pair $(n,r)\in \ints \times \ints$, $E^{X,V}_{n,r}$ is given by 
\begin{equation*}
    E^{X,V}_{n,r} (f) := (\mathrm{ev}_{V^{\otimes r}} \otimes V^{\otimes n}) \circ ((V^*)^{\otimes r} \otimes f \otimes V^{\otimes r}) \circ (\beta_{(V^*)^{\otimes r}} \otimes V^{\otimes r}) \circ (X \otimes \mathrm{coev}_{(V^*)^{\otimes r}}),
\end{equation*}
where we use the convention $V^{\otimes -1} \cong V^*$, and by $\beta_-$, we denote the half-braiding $X\otimes - \to -\otimes X$ of $X$.

\begin{rem}
In the definition of $E^{X,V}_{n,r}$, we omitted the pivotal structure, as well as half-braidings and associativity morphisms. 
Ng and Schauenburg had to be much more careful with their definition, but they showed that the so-defined indicators are preserved under (pivotal) tensor-equivalence (in detail in \cite{ns-basic} for the case $X = \one$). 
Once this has been done, we may write down the definition in the setting of a strict pivotal monoidal category (as we effectively did) and need not worry about the ambiguities of where to insert the structure maps.
\end{rem}

The definition in terms of traces of endomorphisms of hom-spaces can be awkward to work with.
\Cref{prop_trace_homendos} shows that we can express the indicator as the value of the string diagram
\begin{equation} \label{eq_fs_def_basissum}
    \nu_{n,r}^X(V) = \pica{1}{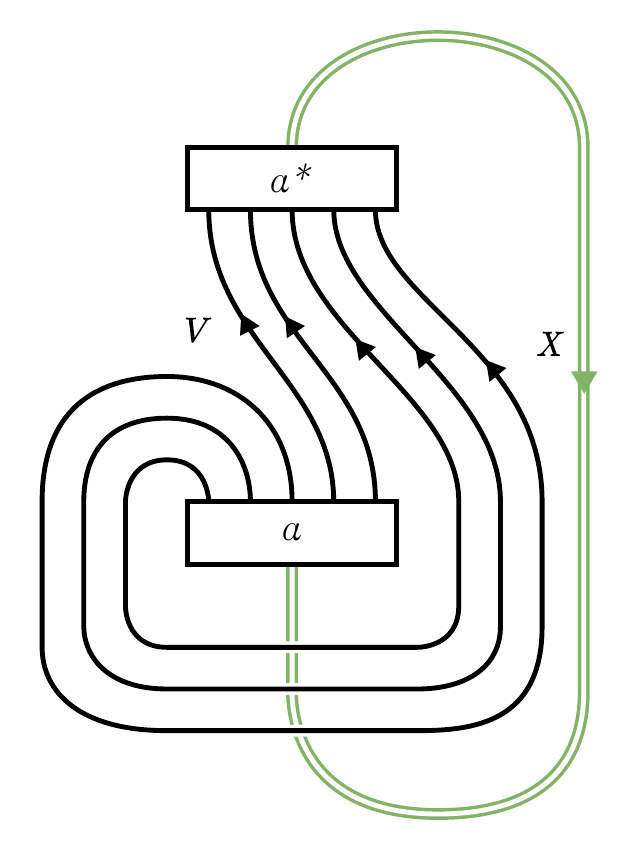},
\end{equation}
which involves a sum over dual bases $\alpha$, $\alpha^*$.
\section{A State-Sum Model for Manifolds with Free Boundaries}\label{sec_tvgen}
In this section, we describe a procedure, based on the Turaev-Viro state sum construction in the formulation of \cite{tv}, to associate a number $|\bord{M}| \in \field$ to a (compact oriented) 3-manifold $\bord{M}$, possibly with boundary, with additional data:
\begin{itemize}
    \item A framed link embedded in the interior of $\bord{M}$, whose components are labeled by objects in $Z(\ccat)$.
    \item A link embedded on the boundary surface of $\bord{M}$, whose components are labeled by objects in $\ccat$. (By a link we here mean an embedding $(\mathbb{S}^1)^{\sqcup n} \to \partial\bord{M}$, i.e. there are no crossings.)
\end{itemize}
In contrast to the usual setting of the Turaev-Viro construction, the boundary $\partial \bord{M}$ does not get assigned a vector space, and it does not make sense to glue together manifolds along boundary components.
We refer to the type of boundary we consider here as a \emph{free} or physical \emph{boundary} -- as opposed to a \emph{gluing boundary}.

In order to compute the scalar $|\bord{M}| \in \field$, an auxiliary structure called a \emph{skeleton} is embedded into the 3-manifold $\bord{M}$ \cite[Sec. 11.5.1]{tv}.
The skeleton comes with a decomposition into finitely many oriented 2-cells, edges between them, and vertices.
The picture \eqref{eq_ball_around_vert} illustrates how a skeleton may locally look around a vertex.
We moreover require that the boundary of $\bord{M}$ is a subset of the skeleton, and that the 3-cells (the connected components of the complement of the skeleton in $\bord{M}$) are open balls.

If the manifold $\bord{M}$ comes equipped with embedded links, we call a skeleton of $\bord{M}$ a \emph{graph skeleton} if the embedded links restrict to the skeleton, and only meet the skeleton's edges transversely, at special vertices called switches. 
Locally, the neighborhood of a switch looks as follows:
\begin{equation*}
    \pica{1}{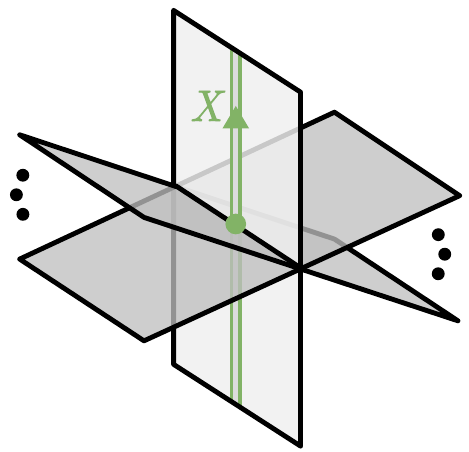}
\end{equation*}
We further require that the orientations of the two 2-cells which a graph intersects in the neighborhood of a switch are compatible (in that they define an orientation on the union of the cells).
The embedded links split some 2-cells into several \emph{faces}: In the picture of the switch above, there are six 2-cells, but eight faces.

We now describe how the number $|\bord{M}|$ is obtained as a state-sum, recalling and slightly extending the construction in \cite[Chap. 15]{tv}.
\begin{enumerate}
    \item Given a 3-manifold $\bord{M}$ with embedded links, pick a graph skeleton $P$.
    For now, fix a \emph{coloring} $c$, that is, a map of sets 
    \begin{equation*}
        c : \mathrm{Fac}(P) \to I,
    \end{equation*}
    where $\mathrm{Fac}(P)$ denotes the set of faces of the skeleton $P$ and $I$ is the chosen set of representatives of simple objects of the spherical fusion category $\ccat$. The objects $c(f)$ are often called \emph{state-sum variables}, as we will later sum over all possible colorings $c$.
    
    \item By a \emph{half-rim} we mean a vertex or a switch of $P$ together with a germ of an adjacent edge of $P$ (in this case we speak of a \emph{half-edge}) or an adjacent component of an embedded link.
    To each half-rim $e$, we associate a vector space $H_c(e)$ as follows.
    If $e$ is a half-edge, the vector space is constructed from the value of the coloring $c$ on the faces adjacent to $e$, their cyclic order around $e$, their orientations, and the orientation of $e$. 
    The following example makes the construction clear. (Here, $Y_1, \dots, Y_4 \in \ccat$ are the objects assigned to the faces by the coloring $c$.)
        \begin{equation}\label{eq_tv_multmod_ex}
            H_c\left(\pica{1}{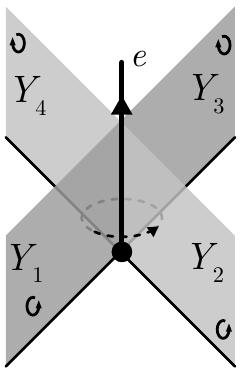}\right) \cong  \homsetin{\ccat}{\one}{Y_{1} \otimes Y_{2}^* \otimes Y_{3} \otimes Y_{4}^*}.
        \end{equation} 
        
    If $e$ is a switch together with a germ of a link segment, then $e$ has precisely two adjacent faces $b_l$ and $b_r$. 
    So locally, a neighborhood of $e$ looks as follows.
        \begin{equation*}
            \pica{1}{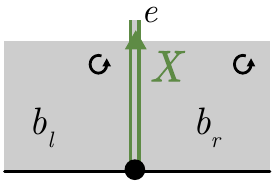}
        \end{equation*}
    The notation $b_l$ and $b_r$ for the faces is not arbitrary: 
    We can always draw a neighborhood of the half-rim $e$ such that the link segment of $e$ points upwards, and such that both faces are oriented positively on the paper, as is done in the picture above. This is possible because the two faces come from a single oriented 2-cell, and thus their orientations are compatible.
    Then it is clear that there is a well-defined notion of a left face (this we denote by $b_l$) and a right face (this we denote by $b_r$).
    The associated vector space is 
        \begin{equation}\label{eq_statespace_insertionrim}
            H_c (e) = \homsetin{\ccat}{\one}{c(b_r) \otimes X^* \otimes c(b_l)^*},
        \end{equation}
    if the half-rim $e$ is oriented away from the switch and
    \begin{equation}\label{eq_statespace_insertionrim1}
            H_c (e) = \homsetin{\ccat}{\one}{c(b_r) \otimes X \otimes c(b_l)^*},
        \end{equation}
    if it is oriented towards the switch.

    We may also assign a vector space to a vertex or switch $v$ in $P$: We set
    \begin{equation}
        H_c (v) = \bigotimes_{e \in \mathbf{R}(v)} H_c(e).
    \end{equation}
    Here, $e$ runs over half-rims adjacent to $v$, and $\otimes$ denotes the (unordered) tensor product of vector spaces.
    
    Finally, we assign a vector space to the 3-manifold $\bord{M}$ (together with the  chosen graph skeleton) as a whole.
    \begin{equation}
        H_c (\bord{M}) = \bigotimes_{e \in \mathbf{R}} H_c (e),
    \end{equation}
    where the product runs over all half-rims.

    \item There is yet another way to decompose the space $H_c(\bord{M})$.
    Note that every rim $r$ (this is an edge of $P$ or a segment of an embedded graph) gives rise to precisely two half-rims $e$, $e'$, which are opposite to one another, so we may rewrite the product over half-rims as a product over internal rims:
    \begin{equation*}
        H_c(\bord{M}) = \bigotimes_{r} H_c(e) \otimes H_c(e').
    \end{equation*}
    With the definitions above, the vector spaces associated to opposite half-rims are in duality:
    \begin{equation}\label{eq_tv_dual_halfrims}
        H_c(e) \cong \left( H_c(e') \right)^*.
    \end{equation}
    Using the duality \eqref{eq_tv_dual_halfrims}, this allows us to define a distinguished non-zero vector 
    \begin{equation}\label{eq_tv_ast_coev}
        \ast_r = \mathrm{coev}(1) \in H_c(e) \otimes H_c(e'),
    \end{equation}
    associated to a rim $r$ whose half-rims are denoted $e$ and $e'$,
    and hence a distinguished vector in $H_c(\bord{M})$:
    \begin{equation}\label{eq_tv_disting_def}
        \ast_c = \otimes_r \ast_r \in H_c (\bord{M}).
    \end{equation}
    
    \item We now assign, to every internal vertex or switch $v$ of $P$, a linear map
    \begin{equation*}
        \Gamma_c(v) : H_c(v) \to \field
    \end{equation*}
    using the evaluation of string diagrams on spheres.
    (In the notation of \cite{tv}, $\Gamma_c(v)$ is written as $\mathbb{F}_\ccat(\Gamma^c_v)$.)
    This requires different treatments for vertices and switches.
    The discussion follows \cite[Sec. 15.5.1]{tv} closely, the formalism is not changed by the presence of free boundaries.
    \begin{itemize}
        \item If $v$ is a vertex of the skeleton $P$, draw a small sphere $B \cong \mathbb{S}^2$ in $\bord{M}$ around $v$.
        If $v$ lies on the boundary, embed a small neighborhood of $v$ into $\reals^3$ and draw the sphere there.
        The intersection $P \cap B$ is a graph on $B$:
        The faces adjacent to $V$ become lines in $P\cap B$, these are the edges of the graph (which we also call $\Gamma_c(v)$).
        The end points of the edges are vertices, and each vertex is located at an intersection point between a rim of $P$ and $B$.
        The orientations of the faces induce orientations of the edges of $\Gamma_c(v)$, and, since the graph skeleton is colored, the edges also inherit colors.
        This is best illustrated in an example:
        \begin{equation}\label{eq_ball_around_vert}
            \pica{1}{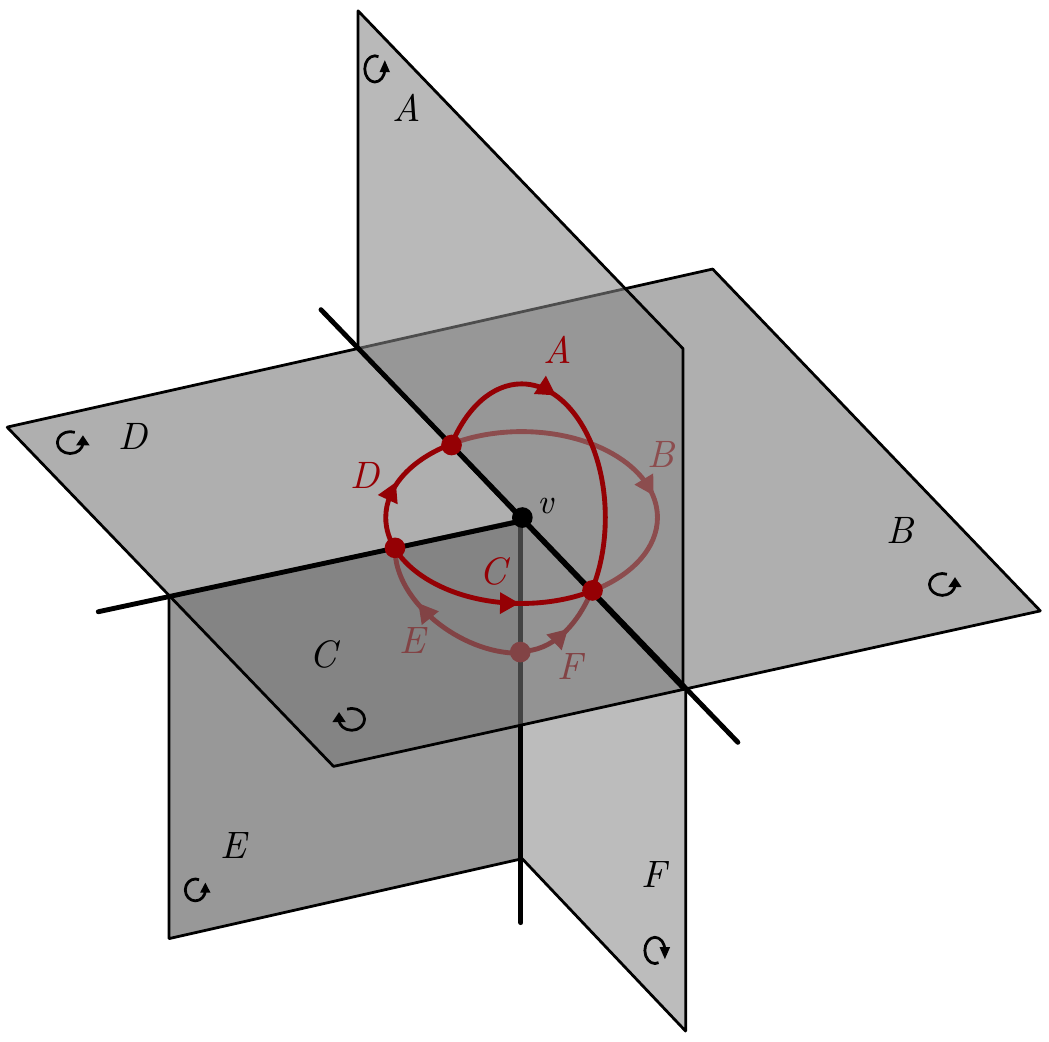}
        \end{equation}
        In the picture, $A$, $B$, $C$, $D$, $E$ and $F$ are elements of $I$ that form a coloring of faces adjacent to the vertex $v$ of a graph skeleton.
        Drawn in red is the graph $\Gamma_c(v)$, which is located on the surface of a small ball around $v$.
        
        Consider a vertex $x$ of $\Gamma_c(v)$ with adjacent edges colored by objects $Y_1, \dots, Y_l$.
        Let $\epsilon_i = -$ if the $i$-th edge is oriented towards the vertex and $\epsilon_i = +$ otherwise.
        We want to read the graph $\Gamma_c(v)$ as a string diagram, so $x$ has to be labeled by a morphism in $\homsetin{\ccat}{\one}{Y_{1}^{\epsilon_1} \otimes \cdots \otimes Y_{l}^{\epsilon_{l}}}$. This space is isomorphic to $H_c(e)$, where $e$ is the half-rim adjacent to $v$ that corresponds to the vertex $x$.
        
        In other words: A choice of vector in $H_c(e)$ for each half-rim $e$ adjacent to $v$ corresponds to a labeling of vertices in the graph $\Gamma_c(v)$ by appropriate morphisms.
        We can then interpret $\Gamma_c(v)$ as a string diagram drawn on a sphere -- which can be evaluated to a scalar in the graphical calculus for the spherical fusion category $\ccat$.
        All together, this yields a map $\Gamma_c(v) : H_c(v) \to \field$.
        
        A more detailed discussion is presented in \cite[Sec. 13.1.1]{tv}.
        
        \item If $v$ is a switch as in Figure \eqref{pic_switch_labels}, then the procedure is the same, except that we add another line (in green) to the graph we evaluate. 
        Since we know the neighborhoods of switches explicitly, we can make this concrete:
        Assume, for the moment, that $v$ does not lie in the boundary of $\bord{M}$. 
        A neighborhood of $v$ looks as follows, where $A,B,C,D,E_1,\dots,E_\mu,F_1,\dots,F_\nu \in I$ for $\mu,\nu \in \nats\cup\{0\}$ and $X \in Z(\ccat)$.
        \begin{equation}\label{pic_switch_labels}
            \pica{1}{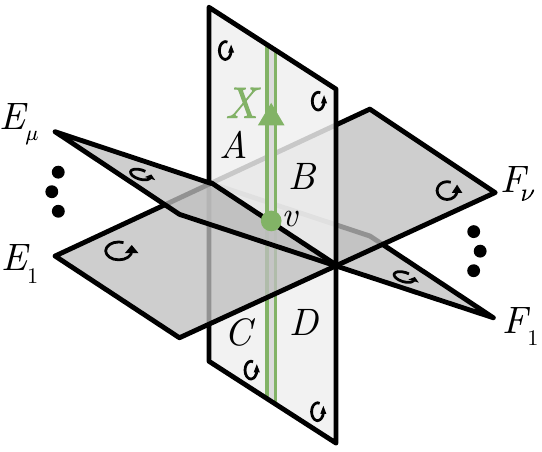}
        \end{equation}
        
        $\Gamma_c(v) : H_c(v) \to \field$ is then given by the map
        \begin{equation}\label{pic_switch_eval}
            \Gamma_c(v) = \pica{1}{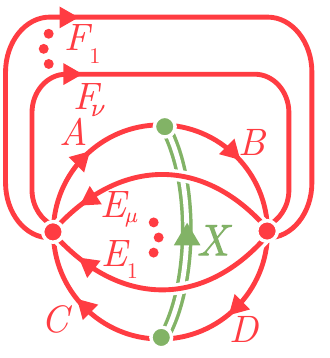},
        \end{equation}
        where again, we interpret a string diagram with unlabeled vertices as a linear map from $H_c(v)$ to $\field$.
        The crossings in the diagram are evaluated using the half-braiding associated with $X \in Z(\ccat)$.
        Of course, the orientations of adjacent faces can be different from those indicated in the picture \eqref{pic_switch_labels}. 
        In this case, the orientations of the corresponding edges in the graph \eqref{pic_switch_eval} change accordingly.
        
        If $v$ does lie on the free boundary, then the defect line is labeled by an object $V \in \ccat$ rather than in the Drinfeld center of $\ccat$ and in the above picture \eqref{pic_switch_labels}, we necessarily have $\mu = 0$.
        This means that in \eqref{pic_switch_eval}, there are no crossings in the diagram, which is important because $V$ does not come with a half-braiding.
    \end{itemize}
    
    \item We now define the scalar $|\bord{M}|$ associated to the manifold $\bord{M}$.
    The sum over $c$ is a sum over all possible colorings.
    \begin{equation}\label{eq_tv_preblock_fixcol}
        \left\lvert \bord{M} \right\rvert := \dim(\ccat)^{-K} \sum_{c} \dim(c)  \left( \otimes_{v} \Gamma_c(v)\right)(\ast_c) \in \field.
    \end{equation}
    Here, $K$ denotes the number of 3-cells in $\bord{M}$, that is, the number of connected components of $\bord{M} \setminus P$.
    The \emph{dimension} of the coloring $c$ is defined as
    \begin{equation}\label{eq_colordim}
        \dim(c) := \prod_{f \in \mathrm{Fac}(d)} \dim(c(f))^{\chi(f)},
    \end{equation}
    where $\chi(f)$ denotes the Euler characteristic of the face $f$.
    In the example we will consider, all faces are topologically disks, so $\chi(f)$ is always 1.
\end{enumerate}

\begin{rem}
    A priori, the quantity $\left\lvert \bord{M} \right\rvert$ depends not only on the manifold data, but also on the choice of graph skeleton $P$. We do not make this explicit in the notation, and posit that it is indeed independent of $P$ (see \Cref{hyp_generaltft}). For $\partial \bord{M} = \emptyset$, this is well-known.
\end{rem}
\begin{rem}
    The formula \eqref{eq_tv_preblock_fixcol} makes clear why we need to work with a \emph{finite} category $\ccat$:
    If there were an infinite number of isomorphism classes of simple objects in $\ccat$, there would also be infinitely many possible colorings, rendering the sum over $c$ ill-defined.
\end{rem}

\section{Computing Frobenius-Schur Indicators from decorated Manifolds} \label{sec_tvfs}

In this section, we will show by direct computation that for a suitable manifold $\bord{T}_{n,r}^{X,V}$ with $V\in \ccat, X\in Z(\ccat)$, the state-sum $|\bord{T}_{n,r}^{X,V}|$ from \Cref{sec_tvgen} is equal to the Frobenius-Schur indicator $\nu_{n,r}^{X}(V)$ from \Cref{defn_fs_general}.
The manifold $\bord{T}_{n,r}^{X,V}$ is a solid torus $\mathbb{D}^2 \times \mathbb{S}^1$, with an embedded $X$-labeled  non-contractible untwisted loop $(-\epsilon,\epsilon) \times \{0\} \times \mathbb{S}^1 \subset \mathbb{D}^2 \times \mathbb{S}^1$, and $n$ $V$-labeled lines on the boundary, which revolve around the torus in such a way that after a full rotation, they connect to the starting point of the line which is $r$ positions away from their own starting point.
For the sake of simplicity, we will here only consider the case $r=1$, where an explicit parametrization of the line labeled by $V$ embedded on the torus surface is given by 
\begin{equation}\label{eq_tvfs_torus_wilsonparam}
    \varphi \mapsto \left(\varphi, \, -n\varphi \right) \in \mathbb{S}^1 \times \mathbb{S}^1,
\end{equation}
for $\varphi \in [0,2\pi)$.
The picture below illustrates $\bord{T}_{n,r}^{X,V}$ for the case $n=4$, $r=1$.
\begin{equation}\label{eq_fs_man_ex}
    \bord{T}_{4,1}^{X,V} = \pica{1}{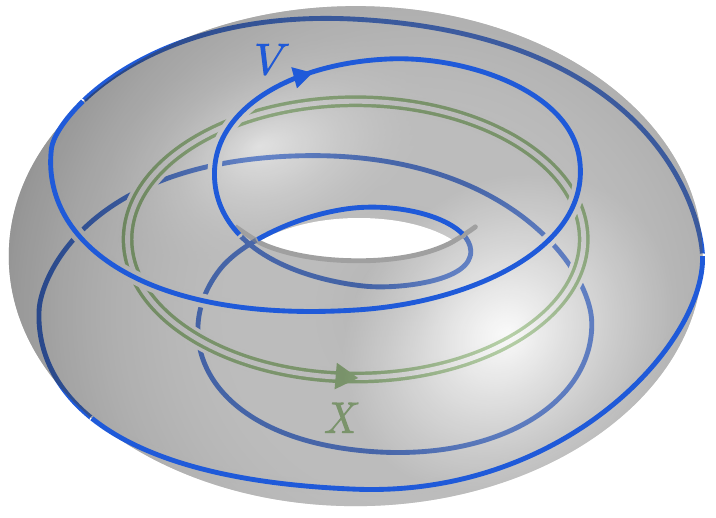}
\end{equation}

We now apply the construction from \Cref{sec_tvgen}.
As a first step, we endow $\bord{T}_{n,1}^{X,V}$ with a graph skeleton $P$.
For drawing purposes, we view the solid torus as a solid cylinder $\mathbb{D}^2 \times [0,1]$ with top and bottom disks identified via $(\xi,0) \sim (\xi,1)$ for $\xi \in \mathbb{D}^2$.
\begin{equation}\label{eq_fs_man_skel}
    \pica{0.7}{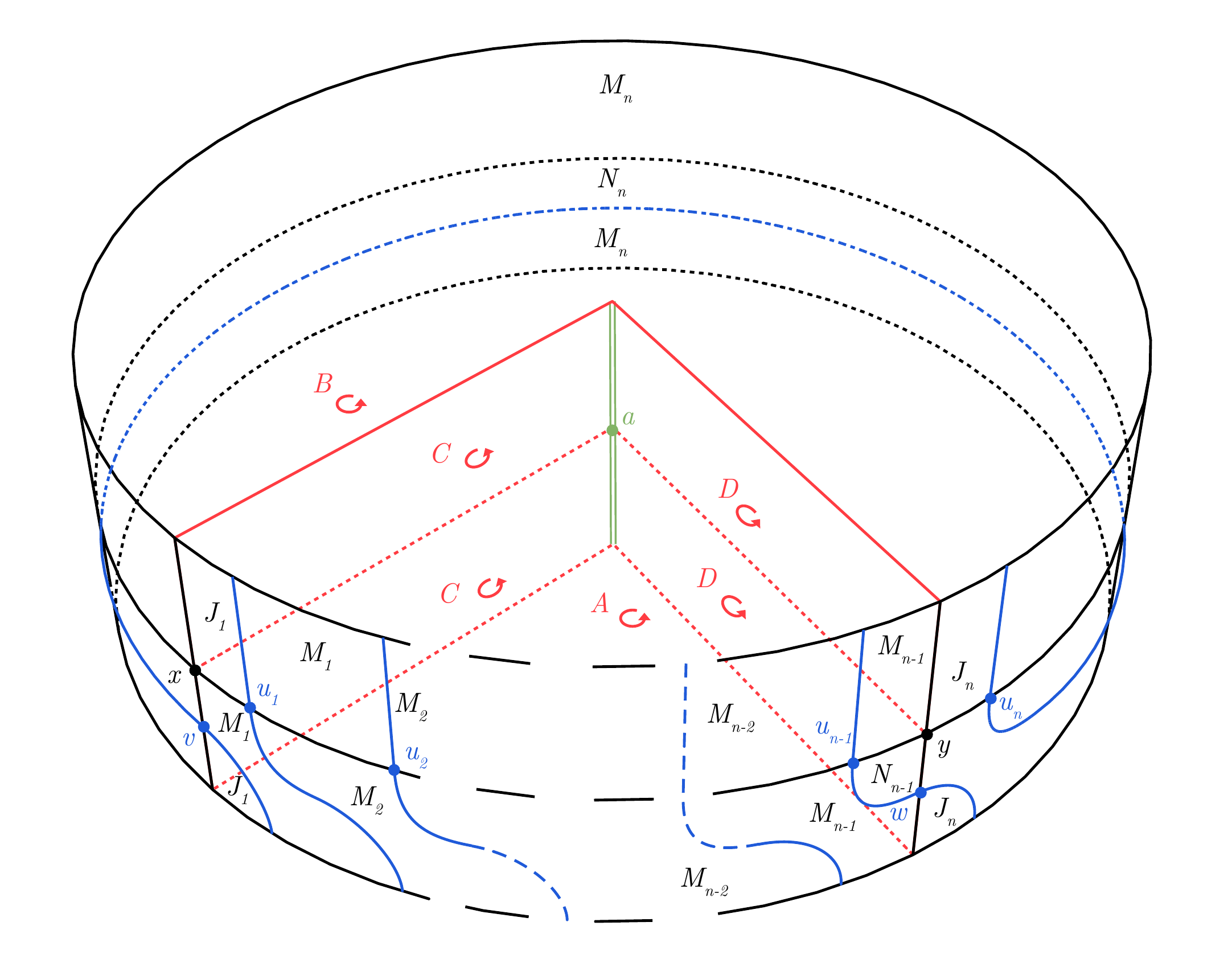}
\end{equation}
We describe the skeletal structure of $P$ in words:

There are $n+5$ vertices and switches in total, among them two black vertices $x,y$ (located on the boundary away from embedded lines), one green switch $a$ (situated on the $X$-labeled line) and $n+2$ blue switches $u_1,\dots,u_n,v,w$ (situated on the $n$ embedded lines on the boundary).

The graph skeleton $P$ has $n+8$ faces. 
Four of them lie in the interior and are drawn in red, labeled $A,B,C,D$.
The horizontal face $A$ is the flat polygon spanned by the vertices and switches $(a,x,u_1,\dots,u_{n-1},y)$.
Similarly, the other horizontal face $B$ is spanned by $(a,y,u_n,x)$.
The faces $C$ and $D$ are perpendicular to $A$ and $B$. 
They are oriented as indicated by the circular arrows.
On the boundary, the faces $J_1,J_n,N_{n-1},N_n$ and $M_i$, $i=1,\dots,n$, are drawn in black and oriented in accordance with the boundary.

Finally, we note that there are two 3-cells.

Similar graph skeleta $P$ exist for all integer values of $n$ and $r$.

\begin{proposition}\label{prop_tvfs}
The 3-manifold $\bord{T}_{n,r}^{X,V}$ with embedded links labeled by objects $X \in Z(\ccat)$ and $V \in \ccat$, endowed with the graph skeleton $P$ described above, evaluates under the state-sum model described in \Cref{sec_tvgen} to the Frobenius-Schur indicator from \Cref{defn_fs_general}:
\begin{equation*}
    |\bord{T}_{n,r}^{X,V}| = \nu_{n,r}^{X}(V).
\end{equation*}
\end{proposition}
The remainder of this section comprises the proof of \Cref{prop_tvfs}.
As mentioned above, we restrict to the case $r=1$. 
This is, however, only a matter of convenience; indeed, \Cref{prop_tvfs} holds for general values of $(n,r) \in \ints \times \ints$.

According to \Cref{sec_tvgen}, we assign a closed graph to each vertex or switch.
These graphs are evaluated on the canonical vector $\ast_c$, weighted by the dimensions of the colors for all faces, and summed over, as dictated by \eqref{eq_tv_preblock_fixcol}.
Here, $c$ denotes a coloring. In the following we will use the same letters for colors of faces as for the faces themselves, writing e.g. $A$ for $c(A)$. 
As a consequence, the sum over $c$ in \eqref{eq_tv_preblock_fixcol} turns into a sum over colors for all faces.
\begin{align}\begin{aligned} \label{twcyl_calc_0}
    |\bord{T}_{n,1}^{X,V}| &=\frac{1}{\dim(\ccat)^2}\sum_{\substack{A,B,C,D\\J_1,J_n,M_\bullet,N_{n-1},N_n}}  \pdim{A}\pdim{B}\pdim{C}\pdim{D}\pdim{J_1}\pdim{J_n}\pdim{M_\bullet}\pdim{N_{n-1}}\pdim{N_n} 
    \\
    & \hspace{.7cm}\Left({1}\Gamma_c(a) \Gamma_c(x) \Gamma_c(y)    \hspace{.3cm}\Gamma_c(u_1) \prod_{i=2}^{n-2}\left( \Gamma_c(u_i) \right)
     \Gamma_c(u_{n-1}) \Gamma_c(u_n) \hspace{.3cm} \Gamma_c(v) \Gamma_c(w) \Right){1}(\ast_c).
\end{aligned}\end{align}
As a first step, we read off all the $n+5$ graphs for the $n+5$ vertices and switches involved from figure \eqref{eq_fs_man_skel}.
This is done using the procedure from \eqref{eq_ball_around_vert} for the vertices.
The graphs associated to switches are obtained using \eqref{pic_switch_eval}.
The vertices of these graphs come in dual pairs, and they have to be labeled by elements of dual bases of the appropriate hom-spaces. (This constructs the vector $\ast_c$.)
We make use of \Cref{conv_dual_verts_in_stringdiags} and do not label the vertices.
The calculations below are easier to follow in color.
With the same order of factors as in \eqref{twcyl_calc_0}, we obtain
\begin{align}\begin{aligned} \label{twcyl_calc_1}
    &\frac{1}{\dim(\ccat)^2}\sum_{\substack{A,B,C,D\\J_1,J_n,M_\bullet,N_{n-1},N_n}}  \pdim{A}\pdim{B}\pdim{C}\pdim{D}\pdim{J_1}\pdim{J_n}\pdim{M_\bullet}\pdim{N_{n-1}}\pdim{N_n} 
    \\
    & \pica{1}{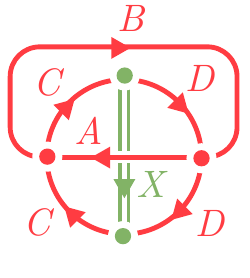} \hspace{1cm}\pica{1}{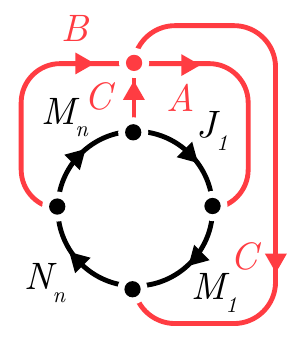} \pica{1}{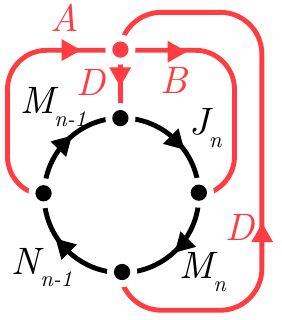} \\
    & \pica{1}{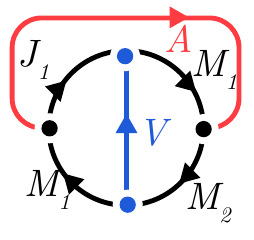}\prod_{i=2}^{n-2}\left(\pica{1}{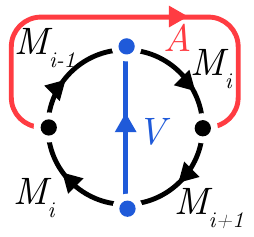}\right)
    \pica{1}{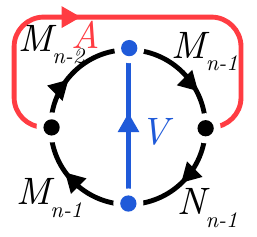}  \pica{1}{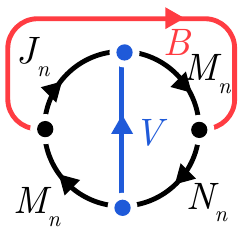}\\
    &  \hspace{1cm}\pica{1}{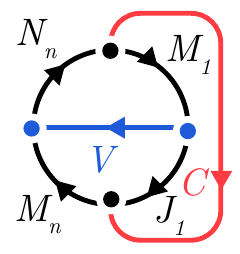}\pica{1}{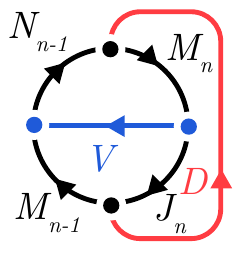}.
\end{aligned}\end{align}
Next, we repeatedly use \Cref{rule_dissolve_verts} to glue together the graphs for the vertices $u_i$, $x$ and $y$.
\begin{align}\begin{aligned} \label{twcyl_calc_2}
    &\frac{1}{\dim(\ccat)^2}\sum_{\substack{A,B,C,D\\J_1,J_n,M_\bullet,N_{n-1},N_n}}  \pdim{A}\pdim{B}\pdim{C}\pdim{D}\pdim{J_1}\pdim{J_n}\pdim{M_\bullet}\pdim{N_{n-1}}\pdim{N_n} 
    \\
    & \pica{1}{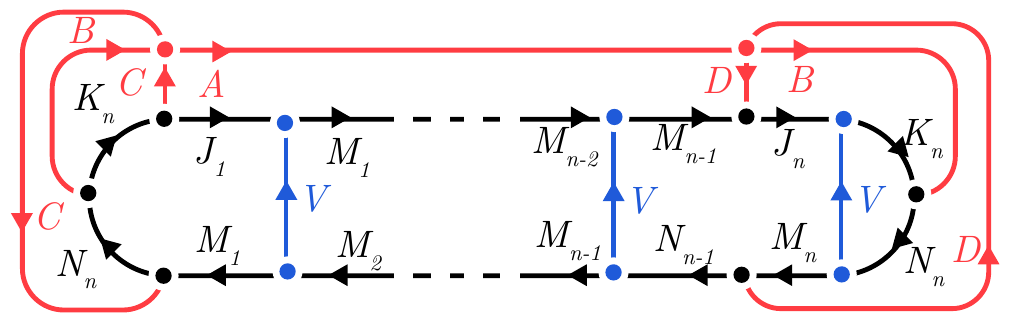}\\
    & \pica{1}{letter_pics/bcalc_1_Gamma_a.pdf} \hspace{1cm}\pica{1}{letter_pics/bcalc_1_Gamma_v.pdf}\pica{1}{letter_pics/bcalc_1_Gamma_w.pdf}.
\end{aligned}\end{align}
\Cref{rule_dissolve_verts} can also be used to glue in the remaining graphs associated to $v$ and $w$.
\begin{align}\begin{aligned} \label{twcyl_calc_3}
    &\frac{1}{\dim(\ccat)^2}\sum_{\substack{A,B,C,D\\J_1,J_n,M_\bullet,N_{n-1},N_n}}  \pdim{A}\pdim{B}\pdim{C}\pdim{D}\pdim{J_1}\pdim{J_n}\pdim{M_\bullet}\pdim{N_{n-1}}\pdim{N_n} 
    \\
    & \pica{1}{letter_pics/bcalc_1_Gamma_a.pdf}\pica{1}{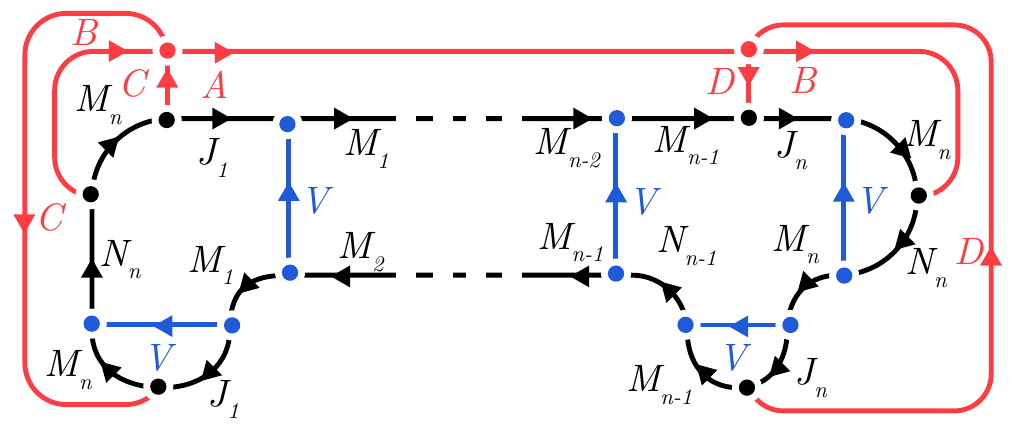}.
\end{aligned}\end{align}
Finally, we use \Cref{rule_dissolve_verts} one last time to obtain a connected graph.
\begin{align}\begin{aligned} \label{twcyl_calc_4}
    &\frac{1}{\dim(\ccat)^2}\sum_{\substack{A,B,C,D\\J_1,J_n,M_\bullet,N_{n-1},N_n}}  \pdim{A}\pdim{B}\pdim{C}\pdim{D}\pdim{J_1}\pdim{J_n}\pdim{M_\bullet}\pdim{N_{n-1}}\pdim{N_n} 
    \\
    & \pica{1}{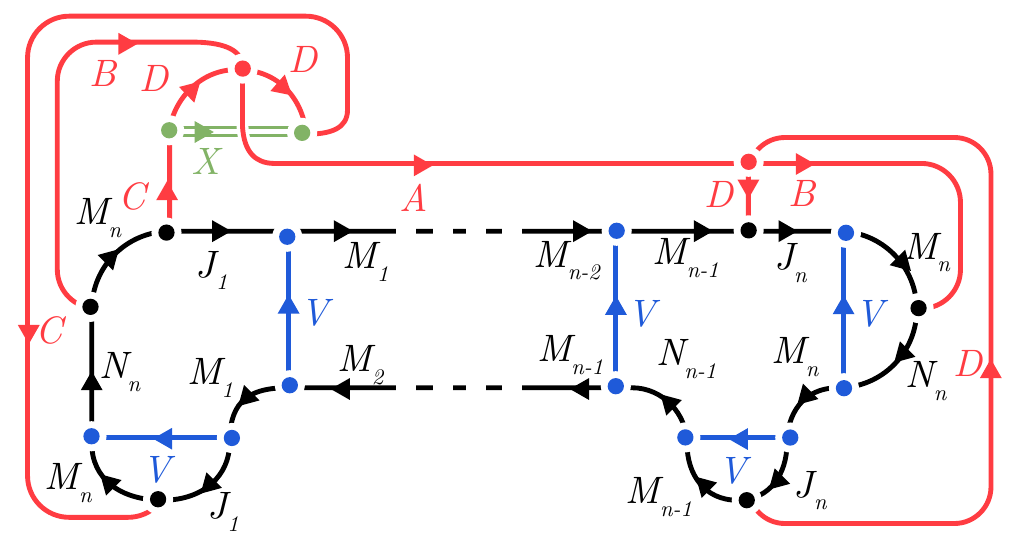}.
\end{aligned}\end{align}
We now use \Cref{rule_dissolve_edges} to dissolve the strand labeled by $N_{n-1}$. In the process, the sum over $N_{n-1}$ (together with the factor $\pdim{N_{n-1}}$) disappears.
\begin{align}\begin{aligned} \label{twcyl_calc_5}
    &\frac{1}{\dim(\ccat)^2}\sum_{\substack{A,B,C,D\\J_1,J_n,M_\bullet,N_n}}  \pdim{A}\pdim{B}\pdim{C}\pdim{D}\pdim{J_1}\pdim{J_n}\pdim{M_\bullet}\pdim{N_n} 
    \\
    & \pica{1}{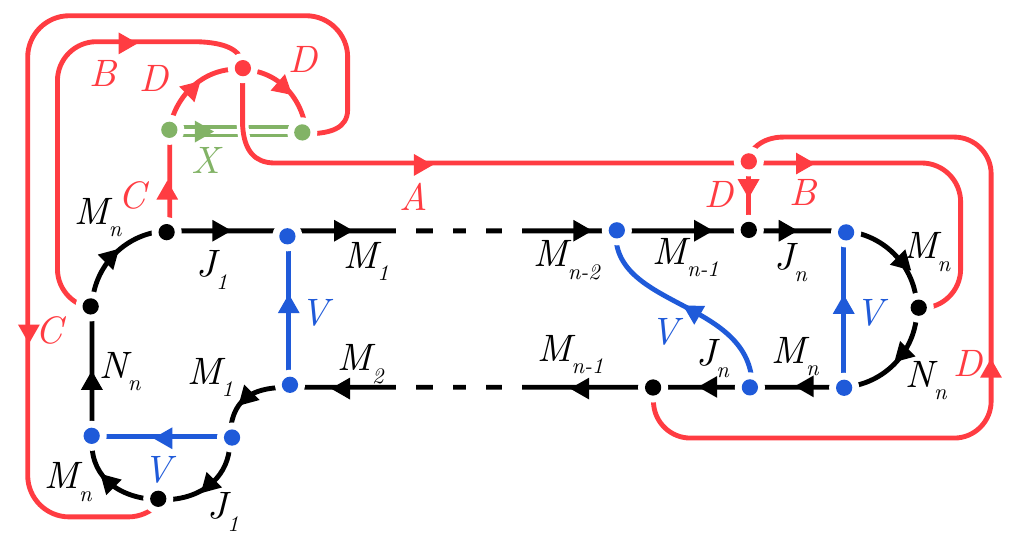}.
\end{aligned}\end{align}
The next step is merely a deformation of the previous graph.
\begin{align}\begin{aligned} \label{twcyl_calc_6}
    &\frac{1}{\dim(\ccat)^2}\sum_{\substack{A,B,C,D\\J_1,J_n,M_\bullet,N_n}}  \pdim{A}\pdim{B}\pdim{C}\pdim{D}\pdim{J_1}\pdim{J_n}\pdim{M_\bullet}\pdim{N_n} 
    \\
    & \pica{1}{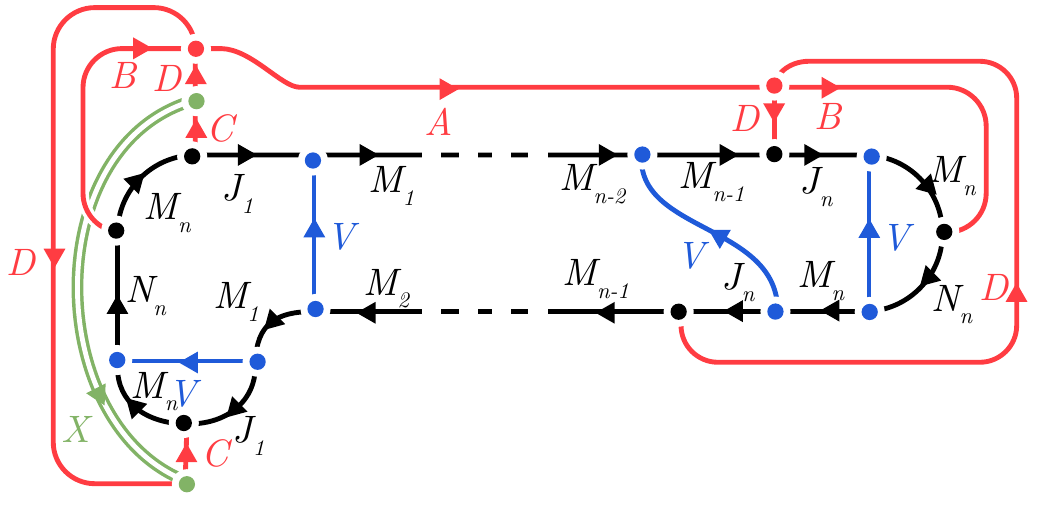}.
\end{aligned}\end{align}
We eliminate the strand labeled by $A$ using \Cref{rule_dissolve_edges}.
\begin{align}\begin{aligned} \label{twcyl_calc_7}
    &\frac{1}{\dim(\ccat)^2}\sum_{\substack{B,C,D\\J_1,J_n,M_\bullet,N_n}}  \pdim{B}\pdim{C}\pdim{D}\pdim{J_1}\pdim{J_n}\pdim{M_\bullet}\pdim{N_n} 
    \\
    & \pica{1}{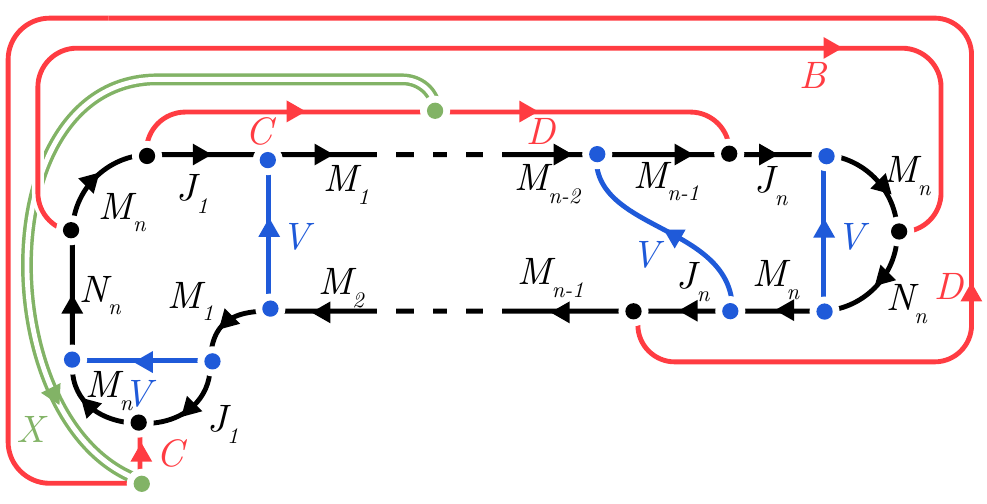}.
\end{aligned}\end{align}
Using the spherical structure, we can flip the $D$-labeled strand from the top to the bottom of the diagram, and then dissolve $B$ using \Cref{rule_dissolve_edges}.
\begin{align}\begin{aligned} \label{twcyl_calc_9}
    &\frac{1}{\dim(\ccat)^2}\sum_{\substack{C,D\\J_1,J_n,M_\bullet,N_n}}  \pdim{C}\pdim{D}\pdim{J_1}\pdim{J_n}\pdim{M_\bullet}\pdim{N_n} 
    \\
    & \pica{1}{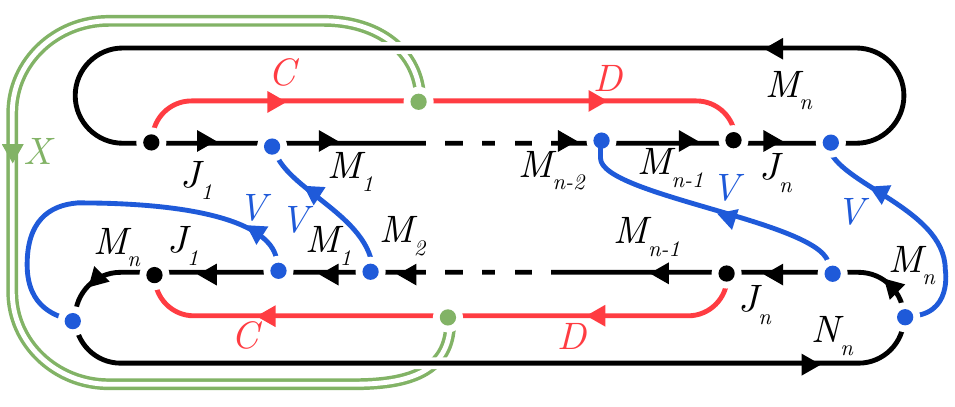}.
\end{aligned}\end{align}
Next, we apply \Cref{rule_dissolve_edges} once more to eliminate $N_n$.
\begin{align}\begin{aligned} \label{twcyl_calc_10}
    &\frac{1}{\dim(\ccat)^2}\sum_{\substack{C,D\\J_1,J_n,M_\bullet}}  \pdim{C}\pdim{D}\pdim{J_1}\pdim{J_n}\pdim{M_\bullet}
    \\
    & \pica{1}{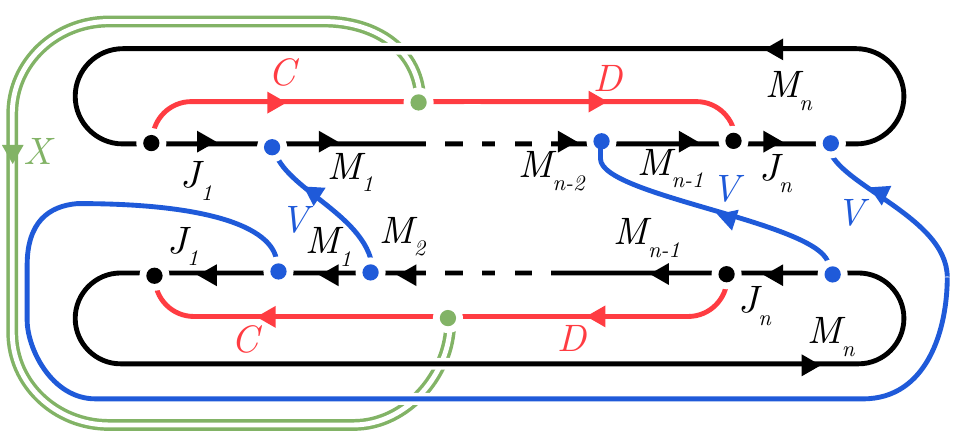}.
\end{aligned}\end{align}
We can now apply \Cref{rule_mergebases} to the pair of lines labeled by $J_1$. The summation over $J_1$ leads to a pair of dual vertices, which we draw as a box. 
\begin{align}\begin{aligned} \label{twcyl_calc_10_a}
    &\frac{1}{\dim(\ccat)^2}\sum_{\substack{C,D\\J_n,M_\bullet}}  \pdim{C}\pdim{D}\pdim{J_n}\pdim{M_\bullet}
    \pica{1}{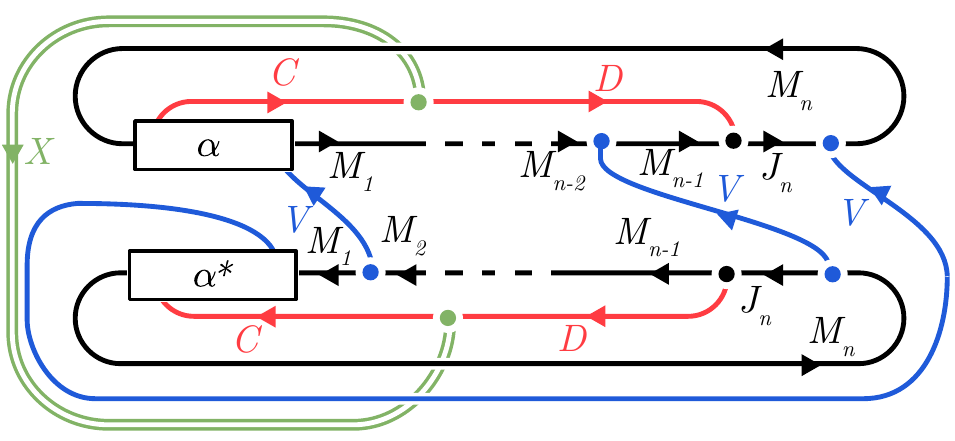}.
\end{aligned}\end{align}
After applying \Cref{rule_mergebases} several more times, only the sums over $C$, $D$ and $M_n$ remain.
\begin{align}\begin{aligned} \label{twcyl_calc_10_b}
    &\frac{1}{\dim(\ccat)^2}\sum_{C,D,M_n}  \pdim{C}\pdim{D}\pdim{M_n}
     \pica{1}{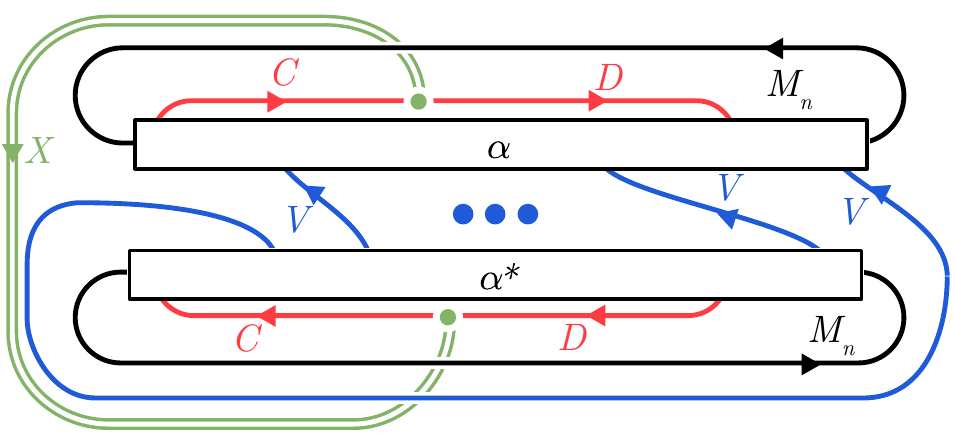}.
\end{aligned}\end{align}
$C$ can also be eliminated using \Cref{rule_mergebases}.
\begin{align}\begin{aligned} \label{twcyl_calc_10_c}
    &\frac{1}{\dim(\ccat)^2}\sum_{D,M_n}  \pdim{D} \pdim{D}\pdim{M_n}
    \pica{1}{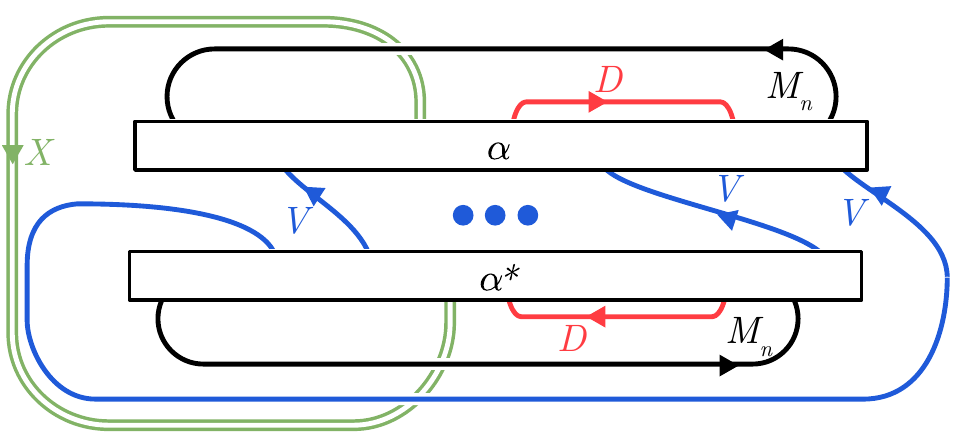}.
\end{aligned}\end{align}
We now employ \Cref{rule_tracedbases} to eliminate the line labeled by $M_n$, together with the sum over $M_n$. 
Note that we obtain a factor of $\dim(\ccat)$.
\begin{align}\begin{aligned} \label{twcyl_calc_10_d}
    &\frac{1}{\dim(\ccat)}\sum_{\substack{D}}  \pdim{D}
     \pica{1}{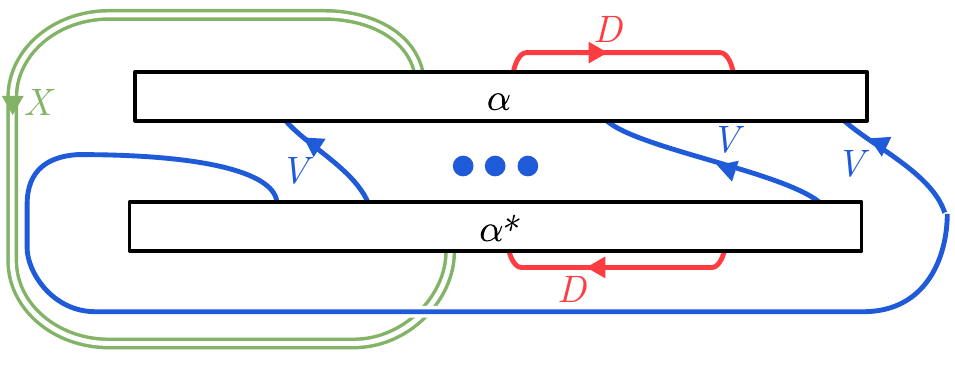}.
\end{aligned}\end{align}
Using \Cref{rule_tracedbases} once more with respect to $D$, we obtain the following graph.
\begin{align}\begin{aligned} \label{twcyl_calc_11}
    & \pica{1}{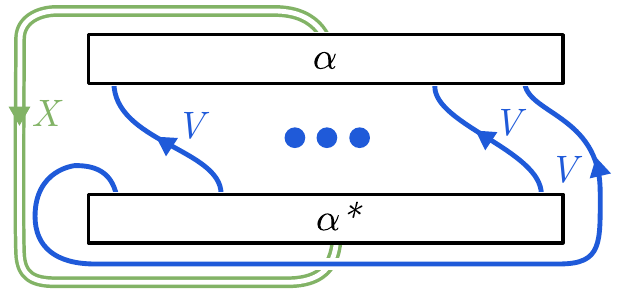}.
\end{aligned}\end{align}
Comparison with \eqref{eq_fs_def_basissum} shows that this is indeed the Frobenius-Schur indicator.

To conclude, we have shown $|\bord{T}_{n,1}^{X,V}| = \nu_{n,1}^{X}(V)$. By a similar calculation, one finds that the general case, $|\bord{T}_{n,r}^{X,V}| = \nu_{n,r}^{X}(V)$ also holds.
\section{$\mathrm{SL}(2,\ints)$-Equivariance}\label{sec_tvfs_equiv}
The geometric representation of the Frobenius-Schur indicators obtained in \Cref{sec_tvfs} provides us with a new tool to study them. 
In this section, we show how one can geometrically understand the $\mathrm{SL}(2,\ints)$-equivariance of the  indicators established by Sommerh\"auser and Zhu \cite{sz-congruencesubgroup}, generalized by Ng and Schauenburg \cite{ns-general}.

It follows from the additivity property of the indicators, $\nu_{n,r}^{X} (V) + \nu_{n,r}^{X'} (V) = \nu_{n,r}^{X\oplus X'} (V)$, that $\nu_{n,r}^{X} (V)$ only depends on the associated vector $[X]\in K_0(Z(\ccat))\otimes_\ints \field$ in the Grothendieck algebra, rather than the object $X$.
Since the Drinfeld center of a spherical fusion category is a modular fusion category, its Grothendieck algebra comes with a $\mathrm{SL}(2,\ints)$-action; this is called the \emph{canonical modular representation} in \cite[Sec. 1.4]{ns-general}.
We may also act with an element of $\mathrm{SL}(2,\ints)$ on the lower indices of the indicator by right matrix multiplication.
For our purpose, we twist the right-multiplication action by multiplying the element $\mathfrak{g} \in \mathrm{SL}(2,\ints)$ from both sides with the invertible matrix $\mathrm{diag}(1,-1)$, denoting
\[\Tilde{\mathfrak{g}} := \left(\begin{array}{cc}
            1 & 0 \\
            0 & -1
        \end{array}\right)\, \mathfrak{g} \left(\begin{array}{cc}
            1 & 0 \\
            0 & -1
        \end{array}\right).\]
The presence of $\mathrm{SL}(2,\ints)$-equivariance means that we can interchange the two actions $[X] \mapsto \mathfrak{g}[X]$ and $(n,r) \mapsto (n,r)\Tilde{\mathfrak{g}}$ in the parameters of the  indicators, which is the content of the following proposition, see \cite[Thm. 8.3]{sz-congruencesubgroup} or \cite[Thm. 5.4]{ns-general}, the former in the context of semisimple Hopf algebras.
\begin{proposition}\label{prop_equiv}
For $\mathfrak{g} \in \mathrm{SL}(2,\ints)$, $(n,r) \in \ints \times \ints$, $V \in \ccat$ and $X \in Z(\ccat)$, we have
\begin{equation*}
    \nu_{(n,r)}^{\mathfrak{g}[X]} (V) = \nu_{(n,r)\Tilde{\mathfrak{g}}}^{[X]} (V).
\end{equation*}
\end{proposition}
The algebraic content of \Cref{prop_equiv} finds its natural geometric interpretation by the result of \Cref{sec_tvfs}.
For the ease of exposition, we work under the following hypothesis, see however \Cref{rem_noproof} for a discussion of the relevance of this assumption.
\begin{hypothesis}\label{hyp_generaltft}
Let $\bordcat$ denote a cobordism category whose objects are compact surfaces, possibly with boundary, with embedded marked points.
The points in the interior of a surface are labeled by objects of $Z(\ccat)$, and the points on the boundary of a surface are labeled by objects of $\ccat$.
The morphisms of $\bordcat$ are (equivalence classes of) 3-manifolds $\bord{M}$ with boundary, an embedded network of $Z(\ccat)$-labeled framed oriented lines in the interior, a decomposition of the boundary $\partial \bord{M}$ into the \emph{gluing boundary} $\partial_g \bord{M}$ and the \emph{free boundary} $\partial_f\bord{M}$, and an embedded network of $\ccat$-labeled oriented lines in the free boundary.

We assume that there exists a (symmetric monoidal) functor
\begin{equation*}
    |-| : \bordcat \to \vect
\end{equation*}
such that 
\begin{itemize}
    \item restricted to the subcategory of $\bordcat$ of cobordisms without free boundary ($\partial_f \bord{M} = \emptyset$), $|-|$ reduces to the usual Turaev-Viro TFT as described in \cite{tv},
    \item in the absence of gluing boundaries ($\partial_g \bord{M} = \emptyset$), $|\bord{M}|$ reduces to the scalar associated to $\bord{M}$ by the construction from \Cref{sec_tvgen}.
\end{itemize}
\end{hypothesis}

The two actions on the parameters of the indicators can be related through the geometric viewpoint by using that $\mathrm{SL}(2,\ints)$ is also the mapping class group of the torus, $\mathrm{MCG(\mathbb{S}^1\times \mathbb{S}^1})$.
From the 3-manifold $\bord{T}_{n,r}^{X,V}$ from \Cref{sec_tvfs}, let us cut out a tubular neighborhood of the inner line labeled by $X \in Z(\ccat)$, and glue this solid torus back into the rest along a homeomorphism $\mathfrak{g} \in \mathrm{MCG(\mathbb{S}^1\times \mathbb{S}^1})$ as pictured and described below.
\begin{equation}\label{glue_tor_inside}
\begin{rcases*}
\hspace{1.2cm}\pica{.7}{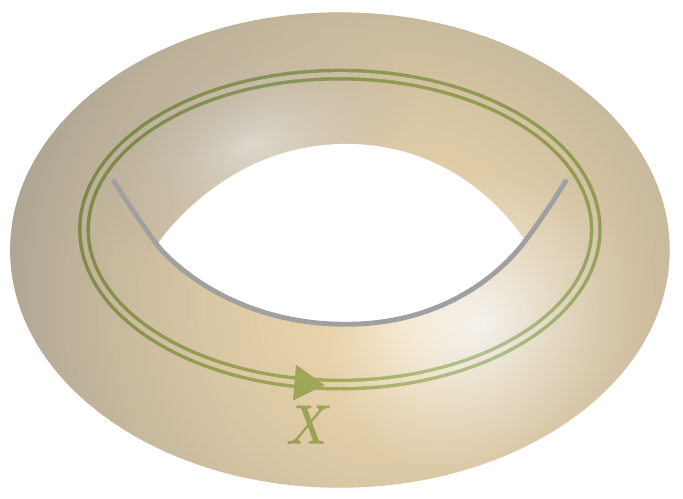} & $= \bord{T}^X$ \\
\pica{.6}{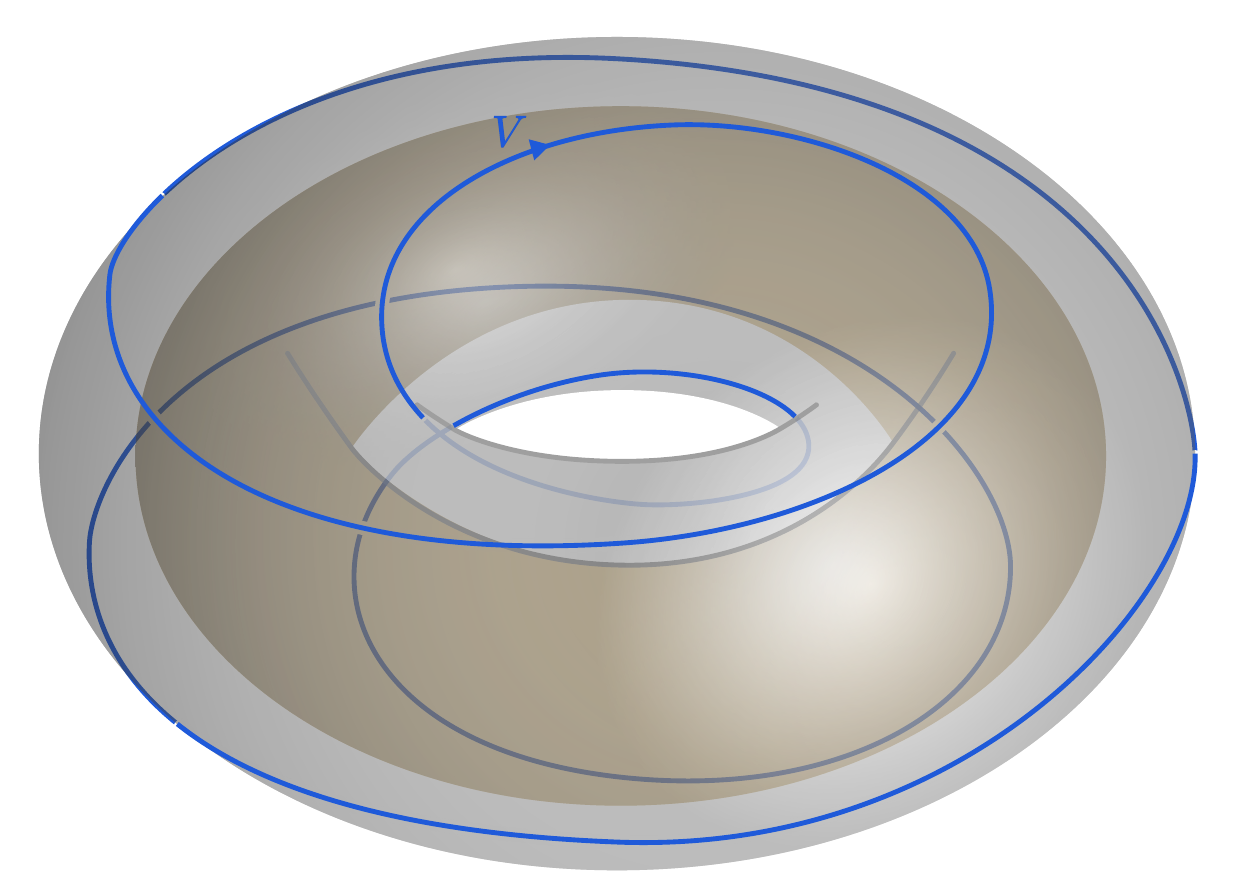} & $= \bord{T}_{n,r}^V$
\end{rcases*}
\quad\xrightarrow[\substack{\text{glue along } \mathfrak{g}\\\text{to a solid torus}}]{}\quad \bord{T}_{n,r}^V \circ_\mathfrak{g} \bord{T}^X.
\end{equation}
The object on the top in the picture \eqref{glue_tor_inside} is a solid torus which we denote by $\bord{T}^X$. 
Its boundary is a gluing boundary, so $\bord{T}^X$ represents in $\bordcat$ a morphism $\emptyset \to \mathbb{S}^1 \times \mathbb{S}^1$.
The manifold on the bottom of Figure \eqref{glue_tor_inside} is a cylinder over a torus, $\bord{T}_{n,r}^V = \mathbb{S}^1 \times \mathbb{S}^1 \times [0,1]$, whose inner boundary component $\mathbb{S}^1 \times \mathbb{S}^1 \times \{0\}$ is a gluing boundary, and whose outer boundary component $\mathbb{S}^1 \times \mathbb{S}^1 \times \{1\}$ is a free boundary with embedded line labeled by $V \in \ccat$.
In $\bordcat$, $\bord{T}_{n,r}^V$ represents a morphism $\mathbb{S}^1 \times \mathbb{S}^1 \to \emptyset$.
It follows from \Cref{hyp_generaltft} and the computation from \Cref{sec_tvfs} that 
\begin{equation*}
    \left\lvert \bord{T}_{n,r}^{X,V} \right\rvert = \left\lvert \bord{T}_{n,r}^V \circ \bord{T}^X \right\rvert,
\end{equation*}
where ``$\circ$'' denotes the gluing of manifolds, which is the composition in $\bordcat$.

It is well-known \cite[Sec. 17.4]{tv} that the vector space associated to a torus is the Grothendieck algebra,
    \begin{equation}\label{eq_tvfs_equiv_torusisgrothendieck}
        \left\lvert \mathbb{S}^1 \times \mathbb{S}^1 \right\rvert \cong K(Z(\ccat)) \otimes_\ints \field,
    \end{equation}
and that under the identification $\homsetin{\vect}{\field}{K(Z(\ccat)) \otimes_\ints \field} \cong K(Z(\ccat)) \otimes_\ints \field$, the manifold $\bord{T}^X$ is assigned the vector $\lvert \bord{T}^X \rvert = [X]$.
Due to the additivity property of the TFT with respect to labels of the embedded loops, 
$\lvert \bord{T}^X \rvert + \lvert \bord{T}^{X'} \rvert = \lvert \bord{T}^{X\oplus X'} \rvert$, it makes sense to label the embedded line in $\bord{T}^X$ not only by an object $X\in Z(\ccat)$, but by a vector $\xi \in K(Z(\ccat)) \otimes_\ints \field$.
We then obtain $\lvert \bord{T}^\xi \rvert = \xi$.

The gluing of manifolds in Figure \eqref{glue_tor_inside} to a solid torus can be realized as a composition of three cobordisms: By a standard construction, an element of the mapping class group $\mathfrak{g} \in \mathrm{SL}(2,\ints) \cong \mathrm{MCG(\mathbb{S}^1\times \mathbb{S}^1})$ defines a cobordism $\bord{T}^\mathfrak{g} : \mathbb{S}^1 \times \mathbb{S}^1 \to \mathbb{S}^1 \times \mathbb{S}^1$.
The gluing of manifolds along $\mathfrak{g}$ as in Figure \eqref{glue_tor_inside} can then be written as a composition of cobordisms $\lvert \bord{T}_{n,r}^V \circ \bord{T}^\mathfrak{g} \circ \bord{T}^X \rvert$.
Our argument relies on the following two equalities of cobordisms:
\begin{equation*}
    \bord{T}_{n,r}^V \circ \bord{T}^\mathfrak{g} = \bord{T}_{(n,r)\Tilde{\mathfrak{g}}}^V, \qquad \lvert \bord{T}^\mathfrak{g} \circ \bord{T}^X \rvert = \lvert \bord{T}^{\mathfrak{g}[X]} \rvert.
\end{equation*}
The second equality is well known \cite[Ex. 4.5.5 and Exc. 4.5.6]{bak-kir}, \cite[Sec. 5.4]{tur-knots}.

To see the first equality, we push the action of the homeomorphism $\mathfrak{g}$ all the way to the boundary, and check that it changes the configuration of the $V$-labeled lines on the boundary accordingly.
It suffices to show this for the two generators $\mathfrak{s}$ and $\mathfrak{t}$ of $\mathrm{SL}(2,\ints)$.
If we draw the surface $\mathbb{S}^1 \times \mathbb{S}^1$ as a rectangle with opposite sides identified, $\mathfrak{s}$ acts as a 90$^\circ$ clockwise rotation.
The following picture shows how the boundary defect lines transform under the $\mathfrak{s}$-transformation:
\begin{equation*}
    \mathfrak{s} :
    \pica{1}{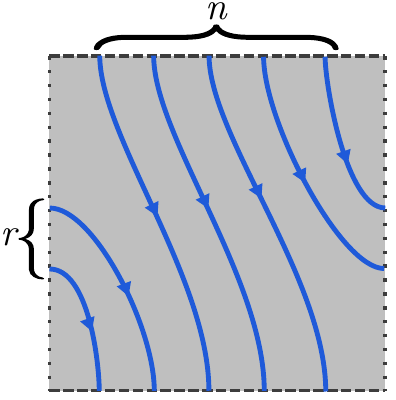} \mapsto \pica{1}{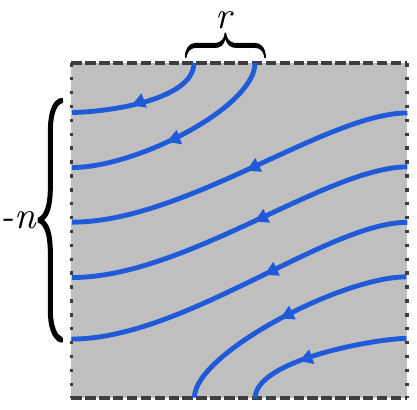}.
\end{equation*}
This is consistent with $(n,r)\Tilde{\mathfrak{s}} = (r,-n)$. Note that the negative sign appears because the strands meet the left side of the rectangle with opposite orientation.
The homeomorphism $\mathfrak{t}$ is a Dehn twist: 
\begin{equation*}
    \mathfrak{t} : 
    \pica{1}{chap_tvfs_pics/FS_equiv_a.pdf}
    \mapsto \pica{1}{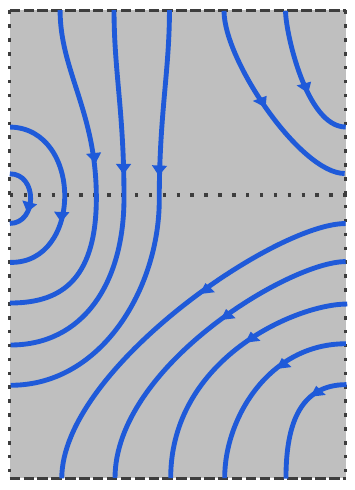} \cong 
    \pica{1}{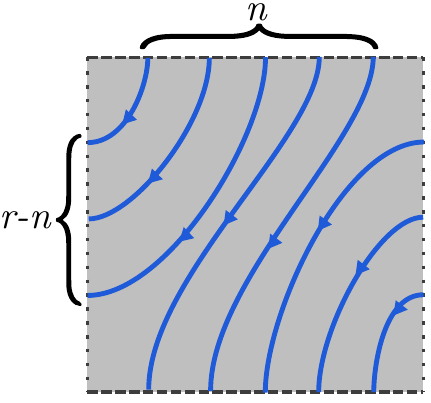}.
\end{equation*}
Here, "$\cong$" means "is isotopic to". In the middle picture, the region above the dashed horizontal line is copied from the left picture, and the region below the dashed line corresponds to the action of a Dehn twist. Again, the result is consistent with $(n,r)\Tilde{\mathfrak{t}} = (n,r-n)$.
    
Combining the results, we find
\begin{equation*}
    \nu_{(n,r)}^{\mathfrak{g}[X]} (V) = \left\lvert \bord{T}^{\mathfrak{g}[X],V}_{n,r} \right\rvert  = \left\lvert \bord{T}^{V}_{n,r} \circ \bord{T}^{\mathfrak{g}[X]} \right\rvert = \left\lvert \bord{T}^{V}_{n,r} \circ \bord{T}^{\mathfrak{g}} \circ \bord{T}^{X} \right\rvert = \left\lvert \bord{T}^{V}_{(n,r)\Tilde{\mathfrak{g}}}  \circ \bord{T}^{X} \right\rvert =  \left\lvert \bord{T}^{X,V}_{(n,r)\Tilde{\mathfrak{g}}} \right\rvert = \nu_{(n,r)\Tilde{\mathfrak{g}}}^{[X]} (V),
\end{equation*}
as in \Cref{prop_equiv}.
This shows the equivariance property from \Cref{prop_equiv}.

\begin{rem}\label{rem_noproof}
As presented here, our arguments rely on \Cref{hyp_generaltft}.
Let us first remark that this is not a bold assumption.
State-sum TFTs with free boundary (or more generally, with defects) have been discussed in the literature \cite{cms-deftfts,peps}.

Moreover, it is possible to give a full proof of equivariance without relying on \Cref{hyp_generaltft}.
To this end, one first introduces a category\footnote{$\bordcatskel$ is not really a category, but a non-unital category without identities: When gluing any two cobordisms in $\bordcatskel$ together, the resulting cobordism will be equipped with the glued skeleton. There is no skeleton which stays unchanged when glued to another skeleton, except for the empty skeleton of the empty manifold.} $\bordcatskel$ of cobordisms with free boundary, where embedded skeleta are part of the data.
The constructions of \Cref{sec_tvgen} extend to a functor $\bordcatskel \to \vect$.
This allows us to argue with cutting-and-gluing operations akin to those in \eqref{glue_tor_inside}, as long as we work with fixed skeleta. 
For a general homeomorphism $\mathfrak{g} \in \mathrm{MCG}(\mathbb{S}^1 \times \mathbb{S}^1)$, the operation \eqref{glue_tor_inside} alters the skeleton.
In order to nevertheless use the argument as above, we need to revert to a skeleton which is invariant under $\mathfrak{g}$. (More precisely, its intersection with the torus boundary needs to be invariant.)
It is not possible to find such a skeleton for general $\mathfrak{g}$, but it is possible to construct an $\mathfrak{s}$-invariant skeleton and an $\mathfrak{st}$-invariant skeleton. (Recall that $\{\mathfrak{s}, \mathfrak{st}\}$ form a set of generators for $\mathrm{SL}(2,\ints)$, and that both elements are of finite order.)
Finally, one then has to adapt the calculations from \Cref{sec_tvfs} to these two skeleta.

Such an extension of our arguments to a new proof of the well-established fact of $\mathrm{SL}(2,\ints)$-equivariance does not substantially add to the insights we gained in this note. We therefore refrain from a presenting such an argument, which will, in any case, be superseded by a general proof of the Turaev-Viro construction with free boundary.
\end{rem}

\paragraph{Acknowledgement}
We thank Peter Schauenburg for a discussion about Frobenius-Schur indicators, and J\"urgen Fuchs and Lukas Woike for helpful comments on the manuscript.
JF is funded by the Deutsche Forschungsgemeinschaft (DFG, German Research Foundation) within the project SCHW\,1162/6-1.
CS is partially supported by the Deutsche Forschungsgemeinschaft under Germany's
Excellence Strategy - EXC 2121 ``Quantum Universe'' - 390833306 and under SCHW\,1162/6-1.
\appendix

\printbibliography

\end{document}